\documentclass[aop,MSNbibl,seceqn,dvips]{arximspdf}

\doi{10.1214/13-AOP852} 
\volume{43}
\issue{1}
\pubyear{2015}
\firstpage{85}
\lastpage{118}

\makeatletter
\newtheorem{teo}{Theorem}[section]
\newtheorem{lem}[teo]{Lemma}
\newproclaim{remark}[teo]{Remark}
\newcommand{\sech}{\operatorname{sech}}
\makeatother

\begin{document}
\begin{frontmatter}

\title{No zero-crossings for random polynomials and the~heat equation}
\runtitle{No zero-crossings for random polynomials}

\begin{aug}
\author[A]{\fnms{Amir} \snm{Dembo}\thanksref{T2}\ead[label=e1]{amir@math.stanford.edu}}
\and
\author[B]{\fnms{Sumit} \snm{Mukherjee}\corref{}\ead[label=e2]{sumitm@stanford.edu}}
\runauthor{A. Dembo and S. Mukherjee}
\affiliation{Stanford University}
\address[A]{Department of Mathematics\\
\quad and Department of Statistics\\
Stanford University\\
Sequoia Hall, 390 Serra Mall\\
Stanford, California 94305-4065\\
USA\\
\printead{e1}}
\address[B]{Department of Statistics\\
Stanford University\\
Sequoia Hall, 390 Serra Mall\\
Stanford, California 94305-4065\\
USA\\
\printead{e2}}
\end{aug}
\thankstext{T2}{Supported in part by NSF Grant DMS-11-06627.}

\received{\smonth{9} \syear{2012}}
\revised{\smonth{4} \syear{2013}}

%
\begin{abstract}
Consider random polynomial $\sum_{i=0}^n a_i x^i$
of independent mean-zero normal coefficients $a_i$,
whose variance is a regularly varying function (in $i$)
of order $\alpha$. We derive general criteria for
continuity of persistence exponents for centered
Gaussian processes, and use these to show that
such polynomial has no roots in $[0,1]$ with probability
$n^{-b_\alpha+o(1)}$, and no roots in $(1,\infty)$
with probability $n^{-b_0+o(1)}$, hence for $n$ even,
it has no real roots with probability $n^{-2b_\alpha- 2b_0+o(1)}$.
Here, $b_\alpha=0$ when $\alpha\le-1$ and otherwise
$b_\alpha\in(0,\infty)$ is independent of the detailed
regularly varying variance function and corresponds to
persistence probabilities for an explicit stationary
Gaussian process of smooth sample path. Further,
making precise the solution $\phi_d({\mathbf x},t)$ to the
$d$-dimensional heat equation initiated by a Gaussian white
noise $\phi_d({\mathbf x},0)$, we confirm that the probability
of $\phi_d({\mathbf x},t)\neq0$ for all $t\in[1,T]$, is
$T^{-b_{\alpha} + o(1)}$, for $\alpha=d/2-1$.
\end{abstract}

%

\end{frontmatter}

\section{Introduction}
Algebraic polynomials of the form
%
%
\begin{equation}
\label{one-rev} Q_n(x)=\sum_{i=0}^n
a_i x^i
\end{equation}
with $x \in\mathbb{R}$ and independent, zero-mean random
coefficients $a_i$ are objects of much interest
in probability theory. In particular, for i.i.d.
normal $\{a_i\}$, the number $N_n$ of real roots
has been studied in some detail, starting with
Littlewood and Offord work \cite{LO1,LO2,LO3}
that provides upper and lower bounds on
$E_n=\mathbb{E}[N_n]$ as well as on both tails
of the law of $N_n$. Among its consequences
is the upper bound
$\mathbb{P}(N_n=0)=O(\frac{1}{\log n})$, much refined
in \cite{DPSZ}, which proved that for $n$ even
$\mathbb{P}(N_n=0)=n^{-4b_0+o(1)}$ decays polynomially
and that the same positive, finite, power
exponent~$b_0$ applies for any i.i.d. $\{a_i\}$
of finite moments of all orders.

In another direction, Kac \cite{K} provides an
explicit formula
for $E_n$ in case of i.i.d. normal $\{a_i\}$,
yielding also the sharp asymptotics
$E_n\sim\frac{2}{\pi}\log n$, whereas
\cite{Mas} shows that $N_n$ is asymptotically normal
of mean $E_n$ and
$\operatorname{Var}(N_n)\sim\frac{4}{\pi}(1-\frac{2}{\pi})\log n$.
Most\vspace*{1pt} of these results extend to other distributions
of the i.i.d. $\{a_i\}$ (see the historical account
in \cite{DPSZ}, Section~2).
We also note in passing the rich asymptotic
theory for location of \textit{complex} zeros
of $z \mapsto Q_n(z)$ and related random
analytic functions (cf. \cite{IbZa,KaZa}
and the references therein).

Our focus here is on persistence probabilities
%
%
\begin{equation}
\label{p0} p_{J}(n)=\mathbb{P}\bigl(Q_n(x)< 0,\ \forall x \in J\bigr).
\end{equation}
Such probabilities have been extensively studied,
for other stochastic processes, also in
reliability theory and in the physics literature,
cf. the surveys \cite{M,AS} and references therein.
Specifically, we study the asymptotics
of $p_J(n)$ for $J=[0,1]$, $J=(1,\infty)$,
$J=[0,\infty)$ and $J=\mathbb{R}$, where $\{a_i\}$ are
independent, centered normal with
$\mathbb{E}(a_0^2)=1$ and
$i \mapsto\mathbb{E}(a_i^2) = i^\alpha L(i)$
forms a regularly
varying sequence of order $\alpha$, at
$i \to\infty$. Equivalently, we consider
any $i \mapsto L(i)$ slowly varying at
infinity (namely, such that
$L([\mu i])/L(i)\to1$ when
$i \to\infty$, for any fixed $\mu>0$, cf. \cite{BGT}).
To this end, deriving in Theorem
\ref{general} a new, general flexible
criteria for continuity of persistence
probability tail exponential rates, we
show in Theorem~\ref{poly} that for
\textit{any} slowly varying $L(\cdot)$,
\begin{eqnarray*}
p_{[0,1]}(n)&=& n^{-b_\alpha+o(1)},\qquad p_{(1,\infty)}(n)=
n^{-b_0+o(1)},
\\
p_{[0,\infty)}(n) &=& n^{- b_\alpha- b_0+o(1)}.
\end{eqnarray*}
Subject to a mild regularity condition on $L(2k)/L(2k+1)$,
we further deduce that $p_{\mathbb{R}}(2n)=n^{-2b_\alpha-2b_0+o(1)}$
[clearly, $p_{\mathbb{R}}(2n+1)=0$ and we note
in passing that $\mathbb{P}(N_n=0) = 2 p_{\mathbb{R}}(n)$].

The power exponent $b_\alpha$ is thus universal,
that is, independent of the specific slowly varying
function $L(\cdot)$, and the asymptotics of
$p_{(1,\infty)}(n)$ is further independent of
the order $\alpha$ of the regularly varying
variance of $a_i$ (as already noted in
\cite{SM} for the case of $L(\cdot) \equiv1$).

\subsection{Nonzero crossings for random polynomials}
Hereafter, let $F(s,t):=\sech((t-s)/2)$,
$\{\widehat{Z}_t,t\geq0\}$ denote the
centered
stationary Gaussian process of covariance function
$\exp\{-(t-s)^2/8)\}$ and for each
$\alpha> -1$, consider the centered Gaussian process
%
%
\begin{equation}
\label{Ydef} Y_t^{(\alpha)} = \frac{\int_0^\infty g_t(r) \,dW_r}{(\int_0^\infty g_t(r)^2 \,dr)^{1/2}},
\end{equation}
where $g_t(r):=r^{\alpha/2} \exp(-e^{-t} r)$ (see
\cite{DPSZ}, (1.4), for $\alpha=0$). We start with
some preliminary facts about these processes and
their persistence exponents.

%
\begin{lem}\label{0}
For any $\alpha> -1$, the $\mathcal{C}^\infty(\mathbb{R})$-valued
stochastic process $t \mapsto Y_t^{(\alpha)}$ of (\ref{Ydef}) has
covariance function $F(s,t)^{\alpha+1}$. Further, its persistence
exponent
%
%
\begin{equation}
\label{bk} b_\alpha:= - \lim_{T\rightarrow\infty}\frac{1}{T}\log
\mathbb{P}\Bigl(\sup_{t\in[0,T]} Y_t^{(\alpha)} \leq
\delta_T\Bigr),
\end{equation}
exists and is independent of the precise choice of
$\delta_T\to0$. These persistence exponents are such
that the nonincreasing $(\alpha+1)^{-1} b_\alpha\uparrow1/2$
when $\alpha\downarrow-1$ and
the nondecreasing $(\alpha+1)^{-1/2} b_\alpha\uparrow\hat{b}_\infty$
when $\alpha\uparrow\infty$, where
$\hat{b}_\infty$ denotes the finite persistence exponent
of $\{\widehat{Z}_t\}$.
\end{lem}

%
\begin{remark} Accurate numerical values are known for
some values of $b_\alpha$ (see \cite{SM} and references therein),
but no analytic prediction for it has ever been given.
The best rigorously proved lower and upper bounds at
$\alpha=0$ are $b_0 \in(1/(4\sqrt{3}),1/4]$, derived in
\cite{Mol}, Proposition~2 and \cite{LS1}, Theorem~3.2,
respectively. From Lemma~\ref{0}, we have that
$b_\alpha$ is between $\sqrt{\alpha+1} b_0$ and $(\alpha+1) b_0$.
Hence, $b_\alpha\in(0,\infty)$ admits the corresponding
lower and upper bounds. It further has linear asymptotics at
$\alpha\downarrow-1$ and square-root growth for $\alpha\to\infty$,
thereby confirming the
predictions of \cite{SM}.
\end{remark}

Here is our first main result.

%
\begin{teo}\label{poly}
Consider random algebraic polynomials $Q_n(\cdot)$
of independent, centered normal coefficients $\{a_i\}$
such that $\mathbb{E}[a_0^2]=1$ and let
$L(i):= i^{-\alpha} \mathbb{E}[a_i^2]$, $i \ge1$, for some $\alpha
\in\mathbb{R}$.
\begin{longlist}[(a)]
\item[(a)] Setting hereafter
$b_\alpha\equiv0$ when $\alpha\le-1$ and
$T_n:=\log n$, we have that for any slowly varying sequence $L(\cdot)$,
%
\begin{eqnarray}
\label{ulb1} \lim_{n\rightarrow\infty}\frac{1}{T_n}\log p_{[0,1]}(n)&=&
-b_{\alpha},
\\
\label{ulb2} \lim_{n\rightarrow\infty}\frac{1}{T_n}\log p_{(1,\infty)}(n) &=&
-b_{0},
\\
\label{p3} \lim_{n\rightarrow\infty}\frac{1}{T_n}\log p_{[0,\infty)}(n)&=&
-b_{\alpha}-b_0.
\end{eqnarray}

\item[(b)] If in addition
%
%
\begin{equation}
\label{rate3} \lim_{n\rightarrow\infty}n \biggl|\frac{L(n+1)}{L(n)}-1 \biggr|=0,
\end{equation}
then further,
%
%
\begin{equation}
\label{lb3} \lim_{n\rightarrow\infty} \frac{1}{T_n}\log p_{\mathbb{R}}(2n) =
-2b_{\alpha}-2b_0.
\end{equation}
\end{longlist}
\end{teo}

%
%
\begin{remark}
The rate condition (\ref{rate3}) is the discrete version of the
condition $x \frac{d}{dx} (\log L(x)) \to0$ as $x \to\infty$.
For example, (\ref{rate3}) holds when $L(x)=(\log x)^\gamma$,
for any $\gamma\in\mathbb{R}$, or
when $L(x)=\exp\{(\log x)^\lambda\}$ for any $|\lambda| < 1$,
but fails in case of the slowly varying $L(n)=1+ n^{-1} (1+(-1)^n)$.
\end{remark}

\subsection{Heat equation initiated by white noise}
Setting $K_t({\mathbf x}):=({4\pi t})^{-d/2}\*\exp\{-\frac{
\|{\mathbf x}\|^2_2}{4t}\}$,
recall that for any \textit{smooth enough} $\psi(\cdot)$, the function
%
%
\begin{equation}
\label{E2}\phi_d({\mathbf x},t)=\int_{\mathbb{R}^d}K_{t}({
\mathbf x}-{\mathbf y})\psi({\mathbf y})\,d{\mathbf y}
\end{equation}
is a classical solution of the $d$-dimensional heat equation
%
%
\begin{equation}
\label{E1} \frac{\partial\phi_d({\mathbf x},t)}{\partial t}=\Delta\phi_d({\mathbf x},t)
\end{equation}
on $\mathbb{D}_0 =\mathbb{R}^d \times(0,\infty)$
with initial condition $\phi_d(\cdot,0) = \psi(\cdot)$.
It is \textit{formally}
argued in \cite{SM} that taking for $\psi(\cdot)$ a centered
Gaussian field
of covariance $\delta_d({\mathbf x}-{\mathbf y})$, should yield by (\ref{E2})
a centered Gaussian field
$\phi_d({\mathbf x},t)$ with covariance
$\mathbb{E}[\phi_d({\mathbf x_1},t)\phi_d({\mathbf x_2},s)]=
K_{t+s}({\mathbf x_1}-{\mathbf x_2})$.
Assuming the existence of such a process, it would have for each
fixed ${\mathbf x}_1={\mathbf x}_2={\mathbf x}\in\mathbb{R}^d$, the
time covariance $K_{t+s}({\mathbf0})$. Thus, taking $\alpha=d/2-1$, it
follows that
\[
\phi_d\bigl({\mathbf x},e^t\bigr)\stackrel{\mathcal{ L}} {=}
\sqrt{K_{2e^t}({\mathbf 0})}Y_t^{(\alpha)}
\]
for $\{Y_t^{(\alpha)}\}$ of Lemma~\ref{0}. Consequently,
%
%
\begin{eqnarray}
\label{E3} \lim_{T\rightarrow\infty} \frac{1}{\log T}\log\mathbb{P}\bigl(
\phi_d({\mathbf x},t)\neq0,\ \forall t\in [1,T]\bigr)&=& -b_\alpha,
\\
\label{E3R} \lim_{R\rightarrow\infty} \frac{1}{R}\log\mathbb{P}\bigl(
\phi_1(x,1)\neq0,\ \forall|x| \le R/2\bigr)&=&- \hat{b}_\infty
\end{eqnarray}
for $b_\alpha$ of (\ref{bk}) and $\hat{b}_\infty$ of Lemma~\ref{0}.
That is, the seemingly unrelated random polynomials
$\{Q_n(x)_{x\in[0,1]}\}$ have the same persistence
power exponent $b_\alpha$ as these solutions
$\{\phi_{2(\alpha+1)}({\mathbf x},t)_{t\in[1,T]}\}$ of the heat equation.

While on a set of full measure the random function
${\mathbf x}\mapsto\psi({\mathbf x})$ is not Lebesgue measurable
[hence the integral (\ref{E2}) ill-defined], we make precise
the notion of solution $\phi_d({\mathbf x},t) \in\mathcal{C}^\infty
(\mathbb{D}_0)$
of (\ref{E1}) such that $\phi_d({\mathbf x},t)$ is a centered
Gaussian field of covariance $K_{t+s}({\mathbf x_1}-{\mathbf x_2})$.
(Added in galleys: after our article was accepted for publication we realized that this is already done in
Section 8 of \cite{BL}.)
Of course, upon rigorously constructing such a field we immediately
get the confirmation of both (\ref{E3}) and (\ref{E3R}).

%
\begin{teo}\label{Heat}
Equip $\mathcal{C}_0=\mathcal{C}^{2,1}(\mathbb{D}_0)$ with the
topology of uniform
convergence on compacts of function and its relevant partial
derivatives of first and \mbox{second order}. There exists a
$(\mathcal{C}_0,\mathcal{B}_{\mathcal{C}_0})$-valued, centered
Gaussian field
$\phi_{d}({\mathbf x},t)$ of covariance function
$C(({\mathbf x}_1,t),({\mathbf x}_2,s))=K_{s+t}({\mathbf x}_1-{\mathbf x}_2)$, which
satisfies (\ref{E1}) on~$\mathbb{D}_0$. Further, $\phi_d \in
\mathcal{C}^\infty(\mathbb{D}_0)$
and for any $0<t_1<t_2$,
%
%
\begin{equation}
\label{E4} \phi_d({\mathbf x},t_2)= \int_{\mathbb{R}^d}K_{t_2-t_1}({
\mathbf x}-{\mathbf y})\phi_d({\mathbf y},t_1)\,d{\mathbf y}.
\end{equation}
\end{teo}

\subsection{Continuity of persistence exponents for Gaussian processes}
The motivation for this work lies in the prediction
of \cite{SM2,SM} for much of our results, but the persistence
asymptotics of Theorem~\ref{poly} has
been rigorously derived before only for i.i.d. $\{a_i\}$
[namely, $\alpha=0$ and $L(\cdot) \equiv1$], where
\cite{DPSZ} relies on an explicitly simple closed
form of $\operatorname{Cov}(Q_n(x),Q_n(y))$ for handling this case.
In contrast, no such closed form expression exist
for $\alpha\neq0$ and especially for $L(\cdot)\not\equiv1$,
henceforth requiring a more delicate treatment of the
covariance in various domains of $x,y$, to which much
of our effort is devoted.

Indeed, beware that the convergence of covariance functions for
smooth centered Gaussian processes [such as $Q_n(\cdot)$],
while implying weak convergence of the corresponding
laws, falls short of relating their large deviations
(and in particular the relevant persistence power
exponents). For example, with $Z$ standard normal
independent of $\{ Y_\cdot^{(\alpha)} \}$, the positive
autocorrelation of the smooth, stationary, centered
Gaussian\vspace*{-2pt} process
$\sqrt{1-\epsilon_n} Y_\cdot^{(\alpha)}+\sqrt{\epsilon_n} Z$
is within $\epsilon_n \to0$ of the autocorrelation of
$\{ Y_\cdot^{(\alpha)}\}$ but for $\epsilon_n \log n \to\infty$,
the corresponding persistence exponent is easily shown to be
$0 \ne b_\alpha$. Our second main result shows that in
contrast, persistence power exponent is continuous
for any collection of centered Gaussian processes whose
maxima over compact intervals converge \textit{pointwise},
\textit{arbitrarily slowly}, to those of the limit process
[see (\ref{c1}) below], provided their nonnegative
auto-correlations satisfy a mild uniform integrability
condition [see (\ref{c2})], and the persistence exponent
of the limiting process is somewhat stable [see (\ref{c3})].

%
\begin{teo}\label{general}
Let $\mathcal{S}$ denote the class of all stationary, autocorrelation functions
$A\dvtx [0,\infty) \mapsto[-1,1]$ with $\mathcal{S}_+$ denoting the
subset of
nonnegative $A \in\mathcal{S}$. For centered stationary Gaussian process
$\{Z_t\}_{t\geq0}$ of autocorrelation $A(s,t)=A(0,t-s) \in\mathcal{S}_+$,
the nonnegative, possibly infinite, limit\vspace*{-1pt}
\[
b(A):=-\lim_{T\rightarrow\infty}\frac{1}{T}\log\mathbb {P}\Bigl(\sup
_{t\in[0,T]}Z_t<0\Bigr),\vspace*{-1pt}
\]
exists. Consider centered Gaussian processes $\{Z_t^{(k)}\}_{t\geq0}$,
$1 \le k \le\infty$ (normalized to have $\mathbb{E} [(Z_t^{(k)})^2]=1$),
of nonnegative autocorrelations $A_k(s,t)$, such that $A_\infty(s,t)
\in\mathcal{S}_+$.
Suppose that the following three conditions hold:\vspace*{-1pt}
%
%
\begin{equation}
\label{c2} 
\limsup_{k,\tau\rightarrow\infty}\,\sup_{s \geq0} \biggl\{
\frac
{\log A_k(s,s+\tau)}{\log\tau} \biggr\} < -1.\vspace*{-3pt}
\end{equation}
%
%
\begin{equation}
\label{c3} \limsup_{M\rightarrow\infty}\frac{1}{M}\log\mathbb{P}
\Bigl(\sup_{t\in
[0,M]}Z^{(\infty)}_t <
M^{-\eta}\Bigr) =- b(A_\infty) \qquad\forall\eta>0\vspace*{-1pt}
\end{equation}
and there exist $\zeta>0$ and $M_1<\infty$ such that for any $z \in
[0,\zeta]$ and $M \ge M_1$,\vspace*{-1pt}
%
%
\begin{eqnarray}
\label{c1} \mathbb{P}\Bigl(\sup_{t\in[0,M]}Z^{(\infty)}_t<z
\Bigr) &\le&\liminf_{k\rightarrow\infty} \inf_{s\geq0}
\mathbb{P}\Bigl(\sup_{t\in[0,M]} Z^{(k)}_{s+t}<z
\Bigr)\nonumber
\\
&\le&\limsup_{k\rightarrow\infty}\, \sup_{s\geq0} \mathbb{P}
\Bigl(\sup_{t\in[0,M]} Z^{(k)}_{s+t}<z\Bigr)
\\
&\le& \mathbb{P}\Bigl(\sup_{t\in[0,M]}Z^{(\infty)}_t \le z \Bigr).\nonumber
\end{eqnarray}
Then
%
%
\begin{equation}
\label{cnc} 
\lim_{k,T \rightarrow\infty} \frac{1}{T}\log
\mathbb{P}\Bigl(\sup_{t\in
[0,T]}Z_t^{(k)}<0
\Bigr) =-b(A_\infty).
\end{equation}
\end{teo}

%
\begin{remark}\label{rmk-Tk} Theorem~\ref{general} only requires that
(\ref{c1}) holds for $z=0$ and $z_M = C M^{-\eta} \downarrow0$. Further,
its proof applies even when $A_k(\cdot,\cdot)$ and
$Z^{(k)}_t$ are defined only on $[0,T^\star_k]$, for some given
$T^\star_k \to\infty$,
with the conclusion (\ref{cnc}) valid then
for any unbounded $T_k \leq T^\star_k$.
We also note in passing that when dealing with \textit{stationary}
$A_k \in\mathcal{S}_+$ for all $k$ large enough, it suffices to consider
only $s=0$ in (\ref{c2}) and~(\ref{c1}), with (\ref{cnc})
implying in particular that, in such setting,
%
%
\begin{equation}
\label{cnc1} \lim_{k\rightarrow\infty}b(A_k)=b(A_\infty).
\end{equation}
\end{remark}

The first of the three conditions of Theorem~\ref{general}, namely
(\ref{c2}), is usually easy to check. Its second condition,
(\ref{c3}), is relatively mild, and in particular applies
whenever $Z^{(\infty)}_t$ of continuous
sample path has decreasing autocorrelation $A_\infty(0,t)$ such that
%
%
\begin{equation}
\label{suff1} a_{h,\theta}^2:= \inf_{0 < t \leq h}
\biggl\{ \frac{A_\infty(0,\theta t)-A_\infty(0,t)}{1-A_\infty(0,t)} \biggr\} > 0
\end{equation}
for any finite $h>0$ and $\theta\in(0,1)$ (see \cite{LS2}, Theorem~3.1(iii), and its proof).

Our next lemma provides explicit sufficient conditions that yield the
last condition, (\ref{c1}), of Theorem~\ref{general}
[and which we utilize when proving Lemma~\ref{0} and part (a)
of Theorem~\ref{poly}].

%
\begin{lem}\label{polyhelp}
Condition (\ref{c1}) holds if to $D \in\mathcal{S}$ corresponds a
Gaussian process
of continuous sample paths and for any finite $M$ there exist positive
$\epsilon_k \to0$ such that whenever $\tau\in[0,M]$ (and $s \in
[0,T^\star_k]$),
%
%
\begin{eqnarray}\label{c11}
(1-\epsilon_k)A_\infty(0,\tau) + \epsilon_k D(0,\tau) &\leq& A_k(s,s+\tau)
\nonumber\\
&\leq& (1-\epsilon_k) A_\infty(0,\tau)+\epsilon_k.
\end{eqnarray}
Alternatively, setting $p^2_k(u):= 2 - 2 \inf_{s \ge0, \tau\in
[0,u]} A_k(s,s+\tau)$,
if $A_k(s,s+\tau) \to A_\infty(0,\tau)$ pointwise and
%
%
\begin{equation}
\label{suff3} \lim_{\delta\downarrow0} \sup_{1 \le k \le\infty} \int
_0^\infty \bigl[p_k
\bigl(e^{-v^2}\bigr) \wedge\delta\bigr] \,d v = 0,
\end{equation}
then the corresponding laws of $\{Z^{(k)}_{s+\cdot}\dvtx  s \ge0, 1 \le k
\le\infty\}$ are uniformly
tight with respect to supremum norm on $\mathcal{C}[0,M]$, which for
$A_k \in\mathcal{S}$
implies that (\ref{c1}) holds for any $z \in\mathbb{R}$.
\end{lem}

For example, by dominated convergence, (\ref{suff3}) holds whenever for
some $\eta> 1$,
%
%
\begin{equation}
\label{suff2} \limsup_{u \downarrow0} |\log u|^{\eta} \sup
_{1 \le k \le\infty} \bigl\{ p^2_k (u) \bigr\} <
\infty.
\end{equation}

%
\begin{remark}\label{rmk-binom}
To demonstrate the flexibility of our approach, we utilize Remark~\ref{rmk-Tk} to confirm
the persistence exponent values predicted by \cite{SM} for the so called
Binomial random polynomials. That is, with $\hat{b}_\infty$ as in
Lemma~\ref{0},
if $\mathbb{E}[a_i^2] = \frac{n}{i}$ for $i=0,\ldots,n$, then
%
%
\begin{eqnarray}
\label{bin1} \lim_{n\rightarrow\infty} n^{-1/2} \log p_{[0,\infty)}(n) &=& - \pi
\hat{b}_\infty,
\\
\label{bin2} \lim_{n\rightarrow\infty} (2n)^{-1/2} \log p_{\mathbb{R}} (2n) &=& - 2
\pi \hat{b}_\infty.
\end{eqnarray}
Indeed, the parameterization $x:=\tan(s/(2\sqrt{n}))$, with $s \in
[0,\pi\sqrt{n})$
for $x \in\mathbb{R}_+$ and $s \in(-\pi\sqrt{n},\pi\sqrt{n})$ in
case $x
\in\mathbb{R}$, translates
the Binomial random polynomials, into stationary, centered Gaussian processes
whose autocorrelations
\[
A_n(s,t):= \biggl[ \cos \biggl( \frac{t-s}{2 \sqrt{n}} \biggr)
\biggr]^n
\]
are nonnegative when either $s,t \in[0, \pi\sqrt{n})$ or $n$ is even.
Recall that the continuous, symmetric function $f(u):= u^2/2 + \log
\cos(u)$
on $|u| \le\pi/2$, decreases in $u \ge0$; hence
$A_n(0,\tau) \uparrow e^{-\tau^2/8}:= A_\infty(0,\tau)$ as $n \to
\infty$,
per fixed $\tau\in\mathbb{R}$ [out of which uniform
super-exponential decay in
$\tau$, hence condition (\ref{c2}) follows]. With
$A_\infty(0,\tau) \in\mathcal{S}_+$ both (\ref{bin1}) and (\ref{bin2}) are
specializations to this context of conclusion (\ref{cnc1}) of
Theorem~\ref{general}, so
it remains only to verify that
(\ref{suff1})~and~(\ref{suff2}) hold here. Now,
condition (\ref{suff1}) holds, for example, by \cite{LS2}, Remark 3.1,
whereas~(\ref{suff2}) holds since $p_n^2(u) \le p_2^2(u) \le u^2/4$
for all $n \ge2$ and~$u$.
\end{remark}

\subsection{Theorem \texorpdfstring{\protect\ref{poly}}{1.3}: Proof outline and extensions}
We proceed to outline the intuition, following \cite{DPSZ} and \cite{SM},
which governs our proof of Theorem~\ref{poly}.
First, since $x\mapsto Q_n(x)$ is continuous, for
$x \in[0,1]$ not too close to $1$, the sign of $Q_n(x)$ can
be controlled by the value of $Q_n(0)$; hence,
the asymptotics of $p_{[0,1]}(n)$ is dominated by the behavior of
$Q_n(x)$ for $x\approx1$. To handle the latter, setting $x=e^{-u}$
allows for approximating
%
%
\begin{equation}
\label{defhaln} \quad\operatorname{Cov}\bigl(Q_n \bigl(e^{-u}
\bigr),Q_n \bigl(e^{-v}\bigr)\bigr)=1+\sum_{i=1}^{n}L(i)
i^{\alpha} e^{-i(u+v)}:= h_{\alpha,n}(u+v)
\end{equation}
for $\alpha> -1$ and small, but not too small values of $u,v$
[namely, in range of $(w_\ell,w_h)$, for
$n w_\ell\to\infty$ and $w_h \to0$], by
\[
\int_{0}^\infty L(r) r^{\alpha} e^{-r(u+v)}\,dr
\sim \Gamma(\alpha+1) (u+v)^{-(\alpha+1)} L\biggl(\frac{1}{u+v}\biggr).
\]
The correlation between $Q_n(e^{-u})$ and $Q_n(e^{-v})$ is
then approximately\break $S(u,v)\* R(u,v)^{\alpha+1}$ where
%
%
\begin{equation}
\label{deflimitcov} R(u,v):= \frac{2\sqrt{uv}}{u+v}, \qquad S(u,v):=\frac{L(1/(u+v))}{\sqrt{L(1/(2u))L(1/(2v))}
}
\end{equation}
and for small $u,v$ the slowly varying nature of $L(\cdot)$ at infinity
implies that $S(u,v)$ is nearly one.
Consequently, replacing $S(u,v)$ by $1$, upon
setting $s:=-\log u$ and $t:=-\log v$ we arrive at the
correlation between $Y_t^{(\alpha)}$ and $Y_s^{(\alpha)}$
with relevant range $t,s \in[\delta T_n,(1-\delta) T_n]$
(for $w_\ell=n^{-(1-\delta)}$ and $w_h=n^{-\delta}$),
yielding
the persistence power exponent $b_\alpha$ of (\ref{bk}).
On a more technical note,
as long as the ratio $u/v$ is bounded, we have indeed
that $S(u,v) \approx1$ for any slowly varying $L(\cdot)$,
but the supremum of $u/v$ over the domain of $(u,v)$
relevant to the asymptotics of $p_{[0,1]}(n)$ is $O(n)$,
requiring us to rely on Theorem~\ref{general}.

Similarly, the main contribution to
$p_{(1,\infty)}(n)$ comes from $x\approx1$. However,
setting $x=e^u$, even at the relevant
range of small $u,v \in(n^{-(1-\delta)},n^{-\delta})$, here
the large values of $i$ dominate the
covariance function of $Q_n(e^{u})$ resulting, for any
$\alpha\in\mathbb{R}$, with
\[
\operatorname{Cov}\bigl(Q_n\bigl(e^{u}\bigr),Q_n
\bigl(e^{v}\bigr)\bigr)= 1+\sum_{i=1}^{n}
L(i) i^{\alpha} e^{i(u+v)} \sim(u+v)^{-1} L(n)
n^{\alpha} e^{n (u+v)}.
\]
The limiting correlation is
now approximately independent of $\alpha$ and $L(\cdot)$,
given for $s=-\log u$ and $t=-\log v$ by $R(u,v) = F(s,t)$
[we note in passing that for \mbox{$\alpha< -1$} this approximation
breaks down at $C(\alpha) \log n/n$, a threshold which
$w_\ell$ must thus exceed, causing further technical
challenge, as seen in proof of Lemma~\ref{ignore}].

Finally, part (b) of Theorem~\ref{poly} then follows
upon showing that, for even values of $n$,
the events of having $Q_n(x)$ negative throughout each of the four intervals
$\pm[0,1]$ and $\pm(1,\infty)$, are approximately independent of
each other
[with (\ref{rate3}) utilized for controlling the dependence between $Q_n(x)$
and $Q_n(-x)$].

%
\begin{remark} We show, in part (b) of Lemma~\ref{ignore}, that
the sequence $n \mapsto p_{[0,1]}(n)$ is bounded away from zero
whenever $\sum_i L(i) i^{\alpha}$ converges
(in~particular, for any $\alpha< -1$). Things are more
involved when $\alpha=-1$, as it is easy to check
that for $L(x)=(\log x)^\gamma$, $\gamma\ge0$ and $n$ large
$h_{-1,n}(e^{-t}+e^{-s})
= (\gamma+1)^{-1} [\min(t,s)]^{\gamma+1}[1 + O(1/\min(t,s))]$
when $t,s \in[1,\log n]$. Hence, for the relevant (large)
values of $t$, the asymptotic autocorrelation of
$Q_n(e^{-e^{-t}})$ is that of Brownian motion, raised to power
$\gamma+1$, suggesting that in this case
$p_{[0,1]}(n) = (\log n)^{-(\gamma+1)/2+o(1)}$ is
sensitive to the choice of $L(\cdot)$. The
lower bound of (\ref{eqsumit-1}) may be improved
to $(|\log v|/|\log u|)^r$, yielding the
persistence lower bound $(\log n)^{-(\gamma+1)+o(1)}$
[by the same reasoning as in proof of (\ref{eqfirst-bm})].
\end{remark}

%
\begin{remark}
As we briefly outline next, Theorem~\ref{general} can also deal
with the main contribution to persistence probabilities for
Weyl random polynomials. Namely, the case of $\mathbb{E}[a_i^2]=1/i!$,
$i \ge0$
and intervals $\overline{J}=[0,\sqrt{n}-\Gamma_n]$ with $\Gamma_n \to
\infty$.
In this setting, we have that
\[
h_n(st):=\operatorname{Cov}\bigl(Q_n(s),Q_n(t)\bigr)=
\sum_{i=0}^n \frac{(st)^i}{i!} \sim
e^{st}
\]
for $s,t \in\overline{J}$, with uniform relative error $\eta_n:= 1 -
e^{-z} h_n(z) = \mathbb{P}(N_z>n)$,
where $N_z$ denotes a Poisson random variable of parameter $z=n - \sqrt {n} \Gamma_n$.
Considering $A_n(s,t):=\operatorname{corr}(Q_n(s),Q_n(t))$ and $A_\infty
(s,t)=e^{-(t-s)^2/2}$,
this yields the bound~(\ref{c11}) for $D(s,t)=A_\infty(s,t)^2$,
some $\epsilon_n \to0$ and all $s,t \in\overline{J}$, so from Lemma~\ref{polyhelp}
we have that (\ref{c1}) holds when $s \in\overline{J}$. The covariance estimate
further implies that
$A_n(s,t)\leq4 A_\infty(s,t)$ for all $s,t \in\overline{J}$ and $n$
large enough,
from which~(\ref{c2}) follows. We have seen already that (\ref{c3})
holds for
$\widehat{Z}_{2t}$ (see Remark~\ref{rmk-binom}),
so taking $n^{-1/2} \Gamma_n \to0$ we deduce from Theorem~\ref{general} that
\[
\lim_{n\rightarrow\infty} n^{-1/2} \log p_{\overline{J}} (n)= - 2
\hat{b}_\infty
\]
as predicted in \cite{SM}. The upper bound
$p_{\mathbb{R}_+} (n) \le\exp(-2 \hat{b}_\infty n^{1/2}
(1+o(1)))$ follows
and to confirm, as predicted there, that it is sharp, one needs only to
show that $n^{-1/2} \log p_{[\sqrt{n}-\Gamma_n,\infty)} (n) \to0$.
\end{remark}

%
\begin{remark} While we do not pursue this here,
by a strong approximation argument like the one done in \cite{DPSZ},
the conclusions of Theorem~\ref{poly}
should extend to nonnormal $\{a_i\}$ with all moments finite.
\end{remark}

%
\begin{remark}
Changing from mean-zero coefficients to regularly varying
negative mean of order $\alpha_\star$ can alter persistence power exponents
associated with $Q_n(\cdot)$, depending on the relation between
$\alpha$
and $\alpha_\star$. Indeed, setting $\mathbb{E}[a_i] = - i^{\alpha
_\star}
L_\star(i)$ for some $\alpha_\star\in\mathbb{R}$, some slowly varying
$L_\star(\cdot)$
and all $i \ge1$, results with $\mathbb{E}[Q_n(e^{-u})]$ having the
same form
as $-h_{\alpha_\star,n}(u)$ in the regime of small, but not too small
values of $u$ of relevance here. The relevant persistence power exponent
is thus reduced, or eliminated all together, when
$h_{\alpha_\star,n}(u) \gg\sqrt{h_{\alpha,n}(2u)}$ and expected to
remain intact when $h_{\alpha_\star,n}(u) \ll\sqrt{h_{\alpha,n}(2u)}$.
The same applies for the persistence power exponents associated with
the neighborhood of $-1$, except for $\mathbb{E}[Q_n(-e^{-u})]$ having the
form of $h_{\alpha_\star-1,n}(u)$, due to cancellations \mbox{between}
mean values for even coefficients and those for odd coefficients.
For example, $p_{[0,1]}(n) = n^{-o(1)}$ even for $\alpha> -1$
as soon as $(\alpha_\star+ 1) > (\alpha+1)/2$, whereas for
$p_{[-1,0]}(n)$ this requires $\alpha_\star>(\alpha+1)/2$. Similarly,
we get the prediction $p_{(1,\infty)} (n) = n^{-\lambda b_0}$ when
$\alpha_\star= (\alpha- \lambda)/2$ for $\lambda\in[0,1]$ [and
upon reducing $\alpha_\star$ by one, same applies for
$p_{(-\infty,-1)}(n)$]. We prove none of these predictions, but
note in passing their agreement in case $\alpha_\star=\alpha=0$
with the rigorous analysis of~\cite{DPSZ}.
\end{remark}

We prove Theorem~\ref{general}, Lemmas~\ref{0}~and~\ref{polyhelp} in
Section~\ref{sec-2}, Theorem~\ref{poly} in
Section~\ref{sec-3} and Theorem~\ref{Heat}
in Section~\ref{sec-5}, devoting Section~\ref{sec-4}
to proofs of the
auxiliary lemmas we use for proving Theorem~\ref{poly}.

\section{Proofs of Lemma \texorpdfstring{\protect\ref{0}}{1.1}, Theorem \texorpdfstring{\protect\ref{general}}{1.6} and
Lemma \texorpdfstring{\protect\ref{polyhelp}}{1.8}}\label{sec-2}

\subsection{Proof of Theorem \texorpdfstring{\protect\ref{general}}{1.6}}
By subadditivity lemma, the existence of the limit
$b(A)$ follows from
Slepian's inequality (see \cite{AT}, Theorem 2.2.1),
and nonnegativity of the autocorrelation $A \in\mathcal{S}_+$.

Considering (\ref{c1}) for $z=0$ and fixed $M$ large enough,
there exist $\xi_k \downarrow0$ such that for all $k$,
\[
\inf_{s\geq0}\mathbb{P}\Bigl(\sup_{t\in[0,M]}Z_{s+t}^{(k)}<0
\Bigr)\geq \mathbb {P}\Bigl(\sup_{t\in[0,M]}Z_t^{(\infty)}<0
\Bigr) - \xi_k.
\]
Thus, by Slepian's inequality and the nonnegativity of $A_k(\cdot,\cdot
)$, we conclude that
\[
\mathbb{P}\Bigl(\sup_{t\in[0,T]}Z^{(k)}_t<0
\Bigr)\geq \Bigl[\mathbb{P}\Bigl(\sup_{t\in
[0,M]}Z^{(\infty)}_t<0
\Bigr) -\xi_k \Bigr]^{\lceil T/M\rceil},
\]
which upon taking $\log$, dividing by $T$ and letting $k,T\rightarrow
\infty$ gives
\[
\liminf_{k,T\rightarrow\infty}\frac{1}{T} \log\mathbb{P}\Bigl(\sup
_{t\in
[0,T]}Z_t^{(k)}<0\Bigr) \geq
\frac{1}{M}\log\mathbb{P}\Bigl(\sup_{t\in[0,M]}
Z^{(\infty)}_t<0\Bigr).
\]
So, considering $M\rightarrow\infty$ completes the proof of the lower
bound in (\ref{cnc}).

To get the matching upper bound, note that by (\ref{c2}), there exist
$\eta>1$ and $M_0$ finite, such that for all large $k$ and any $s,t$,
%
%
\begin{equation}
\label{imme2} A_k(s,t)\leq M_0^\eta|t-s|^{-\eta}.
\end{equation}
For such $\eta$ and $M_0$, set $0<\delta<(1-\eta^{-1})/2$ small
enough for
%
%
\begin{equation}
\label{delta} 4 (M_0 \delta)^\eta\sum_{i=1}^\infty
i^{-\eta} < 1.
\end{equation}
Next, fixing finite $M$ large enough for $\gamma:=(M\delta^2)^{-\eta
}\leq3/4$,
let $s_i=(1+\delta) M i$, $i \ge1$, and consider the
$\delta M$-separated intervals $I_i:= [s_i - M, s_i]$.
Since $|s-t|\geq\delta M |i-j|$ whenever $s\in I_i$, $t\in I_j$,
it follows from (\ref{imme2}) that then
$A_k(s,t)\leq\gamma(M_0 \delta)^\eta|i-j|^{-\eta}$.
Thus, setting $I(t):=i$ for $t\in I_i$ we have that for any \mbox{$s,t \in
\bigcup_i I_i$,}
%
%
\begin{equation}
\label{c4} A_k(s,t) \leq(1-\gamma) A_k(s,t)
1_{\{I(s)=I(t)\}}+ \gamma B\bigl(I(s),I(t)\bigr),
\end{equation}
where $B(i,i)=1$ and $B(i,j):=(M_0 \delta)^\eta|i-j|^{-\eta}$ for $i
\ne j$.
Setting $N:=\lfloor T/(M(1+\delta))\rfloor$ and
\[
\mathcal{J}_T:=\bigcup_{i=1}^N
I_i\subset[0,T],
\]
it follows from (\ref{delta}) and the Gershgorin circle theorem, that
all the eigenvalues of the symmetric $N$-dimensional matrix
${\mathbf B} = \{B(i,j)\}_{i,j=1}^N$ lie within $[1/2,3/2]$. In particular,
${\mathbf B}$ is positive definite
and the RHS of (\ref{c4}) is the autocorrelation of the centered
Gaussian process
$\sqrt{1-\gamma} \overline{Z}^{(k)}_t+\sqrt{\gamma}X_{I(t)}$
on $\mathcal{J}_T$, where the centered, stationary, Gaussian sequence
$\{X_i\}_{i=1}^\infty$ of
autocorrelation $B(i,j)$, is independent of the mutually independent
restrictions of $\overline{Z}^{(k)}_t$ to intervals $I_i$, having
the same law as $Z^{(k)}_t$ within each $I_i$. Thus, by Slepian's inequality
for some $\xi_k \downarrow0$, any $k$ large enough and all $T$,
%
%
\begin{eqnarray}
\label{nc4} \qquad\mathbb{P}\Bigl(\sup_{t\in[0,T]}Z_t^{(k)}<0
\Bigr) &\leq& \mathbb{P}\Bigl(\sup_{t\in
\mathcal{J}_T}Z^{(k)}_t<0
\Bigr)\nonumber
\\
&\leq&\mathbb{P}\Bigl(\sup_{t\in[0,T]}\bigl\{\sqrt{1-\gamma}
\overline{Z}_t^{(k)}+\sqrt{\gamma}X_{I(t)}\bigr\}<0
\Bigr)\nonumber
\\
& =& \mathbb{E} \Biggl[ \prod_{i=1}^N\mathbb{P}
\biggl(\sup_{t\in I_i} Z_t^{(k)}\leq-
\frac{\sqrt{\gamma}}{\sqrt{1-\gamma}}X_i\bigg|{\mathbf X} \biggr) \Biggr]\nonumber
\\
&\leq& \mathbb{E}\prod
_{i=1}^N \Bigl[\mathbb{P}\Bigl(\sup
_{t\in
I_i}Z_t^{(k)}<2\gamma^{\delta}
\Bigr)+1_{\{X_i\leq- \gamma^{\delta-1/2}\}
} \Bigr]
\nonumber
\\
&\leq&\mathbb{E}\prod_{i=1}^N \Bigl[
\mathbb{P}\Bigl(\sup_{t\in
[0,M]}Z_t^{(\infty)} \le2
\gamma^{\delta}\Bigr)+\xi_k + 1_{\{X_i\leq-
\gamma ^{\delta-1/2}\}} \Bigr],
\end{eqnarray}
where in the last inequality we use (\ref{c1}) for $z=2\gamma^\delta
\le\zeta$ (provided $M$ is large enough).
Since $B(i,j)$ is nonincreasing in $|i-j|$, by Slepian's inequality the
last term is in turn further bounded above by
%
\begin{equation}
\label{c5} \sum_{j=0}^N \pmatrix{N\cr j}  \Bigl(
\mathbb{P}\Bigl(\sup_{t\in[0,M]}Z_t^{(\infty)}< 3
\gamma^{\delta
}\Bigr)+\xi _k \Bigr)^{N-j}\mathbb{P}
\bigl(X_i\geq\gamma^{\delta-1/2},1\leq i\leq j\bigr).\hspace*{-25pt}
\end{equation}

Proceeding to bound $\mathbb{P}(X_i\geq\gamma^{\delta-1/2},1\leq
i\leq
j)$, recall that all eigenvalues of~${\mathbf B}$
lie within $[1/2,3/2]$, and so the quadratic form ${\mathbf x'}{\mathbf B}^{-1}
{\mathbf x}$ is bounded bellow by~$\frac{2}{3} \|{\mathbf x}\|_2^2$, yielding
the bound
\begin{eqnarray*}
\mathbb{P}\bigl(X_i\geq\gamma^{\delta-1/2},1\leq i\leq j\bigr)&= &
\operatorname{det}({\mathbf B})^{-1/2} (2\pi)^{-j/2} \int
_{[\gamma^{\delta -1/2},\infty )^j} e^{-(1/2){\mathbf x}'{\mathbf B}^{-1} {\mathbf x}} \,d{\mathbf x}\nonumber
\\
&\leq&\frac{2^{j/2}}{(2\pi)^{j/2}} \int_{[\gamma^{\delta
-1/2},\infty )^j} e^{-1/3\|\mathbf x\|_2^2} \,d{\mathbf x}
\\
&=& 3^{j/2}\mathbb{P}\bigl(X_1\geq\sqrt{2/3} \gamma^{\delta-1/2}\bigr)^j.\nonumber
\end{eqnarray*}
Combining this with (\ref{nc4}) and (\ref{c5}), we deduce that
\[
\mathbb{P}\Bigl(\sup_{t\in[0,T]}Z^{(k)}_t<0
\Bigr)\leq \Bigl[\mathbb{P}\Bigl(\sup_{t\in
[0,M]}Z_t^{(\infty)}<3
\gamma^\delta\Bigr)+\xi_k +\sqrt{3}\mathbb{P}
\bigl(X_1\geq\sqrt{2/3}\gamma^{\delta-1/2}\bigr)
\Bigr]^N.
\]
Considering $T^{-1} \log$ of this inequality in the limit $T,k
\rightarrow\infty$ results with
%
%
\begin{eqnarray}
\label{c6} &&\limsup_{k,T \rightarrow\infty} \frac{1}{T}\log\mathbb {P}\Bigl(\sup
_{t\in[0,T]}Z^{(k)}_t<0\Bigr)
\nonumber\\
&&\qquad \leq\frac{1}{M(1+\delta)}\log \Bigl[\mathbb{P}\Bigl(\sup_{t\in
[0,M]}Z_t^{(\infty)}<
3 \gamma^\delta\Bigr)+\sqrt{3}\mathbb{P}\bigl(X_1\geq \sqrt
{2/3}\gamma^{\delta-1/2}\bigr) \Bigr].\hspace*{-25pt}
\end{eqnarray}
Next, note that with $X_1$ a standard normal variable and $\eta
(1-2\delta) > 1$,
\[
\limsup_{M\rightarrow\infty}\frac{1}{M}\log\mathbb {P}\bigl(X_1
\geq \sqrt{2/3}\gamma^{\delta-1/2}\bigr)\leq-\liminf_{M\rightarrow\infty}
\bigl(3M\gamma^{1-2\delta}\bigr)^{-1} =-\infty,
\]
whereas by (\ref{c3}) we have
\[
\limsup_{M\rightarrow\infty}\frac{1}{M}\log\mathbb {P}\Bigl(\sup
_{t\in[0,M]}Z^{(\infty)}_t< 3 \gamma^{\delta}
\Bigr) = - b(A_\infty).
\]
Thus, considering the RHS of (\ref{c6}) as $M\rightarrow\infty$, then
$\delta\downarrow0$, yields the upper bound in (\ref{cnc}).

\subsection{Proof of Lemma \texorpdfstring{\protect\ref{polyhelp}}{1.8}}
Let $V_t$ denote the stationary, centered Gaussian process of
auto-correlation $D(\cdot,\cdot) \in\mathcal{S}$.
Assuming without loss of generality that $\epsilon_k \in[0,3/4]$ (so
$1-\sqrt{1 -\epsilon_k} \le\sqrt{\epsilon_k} \wedge1/2$),
per fixed $M$ and $z$, by Slepian's inequality and the LHS of (\ref{c11}), for any $s \ge0$ and $k$,
\begin{eqnarray*}
&& \mathbb{P}\Bigl(\sup_{t\in[0,M]}Z_{s+t}^{(k)}< z\Bigr)
\\
&&\qquad \geq\mathbb{P}\Bigl(\sup_{t\in[0,M]} \bigl\{ \sqrt{1-
\epsilon_k}Z_t^{(\infty)} +\sqrt{
\epsilon_k} V_t \bigr\} < z\Bigr)
\\
&&\qquad \geq \mathbb{P}\Bigl(\sup_{t\in[0,M]}Z_t^{(\infty)}<
z - 2 \epsilon _k^{1/4}\Bigr)-\mathbb{P}\Bigl(\sup
_{t\in[0,M]} V_t \ge\epsilon_k^{-1/4}
- |z|\Bigr).
\end{eqnarray*}
By sample path continuity, $\sup_{t\in[0,M]} V_t$ is finite almost
surely, so with $\epsilon_k \to0$ it follows from the preceding that
for any $z$ and $M$ finite,
\[
\liminf_{k\rightarrow\infty} \inf_{s \ge0} \mathbb{P}\Bigl(
\sup_{t\in[0,M]}Z_{s+t}^{(k)}< z\Bigr) \geq
\mathbb{P}\Bigl(\sup_{t\in[0,M]}Z_t^{(\infty)} < z
\Bigr).
\]
Similarly, from the RHS of (\ref{c11}) we have that for any $s \ge0$
and $k$,
\begin{eqnarray*}
&& \mathbb{P}\Bigl(\sup_{t\in[0,M]} Z_{s+t}^{(k)} <
z\Bigr)
\\
&&\qquad \leq \mathbb {P}\Bigl(\sup_{t\in[0,M]} \bigl\{ \sqrt{1-
\epsilon_k} Z_t^{(\infty)} + \sqrt {
\epsilon_k} X_1 \bigr\} < z\Bigr)
\\
&&\qquad \leq \mathbb{P}\Bigl(\sup_{t\in[0,M]}Z_t^{(\infty)}
< z + 2 \epsilon _k^{1/4}\Bigr) + \mathbb{P}
\bigl(X_1 \leq-\epsilon_k^{-1/4} +|z|\bigr),
\end{eqnarray*}
hence for any $z$ and $M$ finite,
\[
\limsup_{k\rightarrow\infty} \sup_{s \ge0} \mathbb{P}\Bigl(
\sup_{t\in[0,M]}Z_{s+t}^{(k)}< z\Bigr) \leq
\mathbb{P}\Bigl(\sup_{t\in[0,M]} Z_t^{(\infty)} \le z
\Bigr).
\]

Turning to the second part of the lemma, recall \cite{AT}, Theorem~1.4.1, that for some
universal constant $C$ and all $s$, $M$, $k$ and $\delta>0$,
\[
\mathbb{E} \Bigl[\sup_{|t-t'| \le\delta, t,t' \le M} \bigl|Z^{(k)}_{s+t}-Z^{(k)}_{s+t'}\bigr|
\Bigr] \leq C \int_0^\infty\bigl[p_k
\bigl(e^{-v^2}\bigr) \wedge\delta\bigr] \,d v
\]
(using integration by parts, one easily confirms that the preceeding is
equivalent to
\cite{AT}, (1.4.5)). Thus, as $Z_s^{(k)}$ has a standard normal law,
for any $k$,
the condition~(\ref{suff3}) guarantees (by an application of
Arzela--Ascoli theorem),
the stated uniform tightness of the laws of $Z_{s+\cdot}^{(k)}$ on
$\mathcal{C} [0,M]$.
As such, by Prohorov's theorem it is a precompact collection of laws
(with respect to
weak convergence on $\mathcal{C}[0,M]$). Clearly, pointwise
convergence of
$A_k(s,s+\tau)$ to $A_\infty(0,\tau)$ implies, per fixed $s$ and finite
$M$, convergence
as $k \to\infty$ of the f.d.d. of $Z^{(k)}_{s+\cdot}$ on $[0,M]$ to
those of
$Z^{(\infty)}_\cdot$. In combination with the preceding
precompactness, this
verifies the convergence of $Z^{(k)}_{s+\cdot}$ to $Z^{(\infty
)}_\cdot$ in
distribution on $\mathcal{C} [0,M]$ (per $s$ and $M$). The convergence
in law of
$\sup_{t \in[0,M]} Z^{(k)}_{s+t}$ to $\sup_{t \in[0,M]} Z^{(\infty)}_t$
which follows (by continuity of $z_\cdot\mapsto\sup_{t \in[0,M]}
z_t$ on
$\mathcal{C} [0,M]$), implies, by definition, the validity of~(\ref{c1})
in case $A_k \in\mathcal{S}$ (where such convergence is by default
uniform in $s$).

\subsection{Proof of Lemma \texorpdfstring{\protect\ref{0}}{1.1}}
The centered Gaussian process $Y_t^{(\alpha)}$ of (\ref{Ydef}) is
well defined
[since the nonrandom, nonzero $g_t \in L_2(\mathbb{R}_+)$ for all $t
\in\mathbb{R}$
and $\alpha>-1$]. Further, since
$\|g_t\|_2 = e^{t (\alpha+1)/2} \|g_0\|_2$ and
\[
(g_t,g_s):= \int_0^\infty
g_t(r) g_s(r) \,dr = \biggl(\frac{e^{-t} + e^{-s}}{2}
\biggr)^{-(\alpha+1)} \|g_0\|_2^2,
\]
it follows that
\[
\operatorname{Cov}\bigl(Y_t^{(\alpha)},Y_s^{(\alpha)}
\bigr) = \frac{(g_t,g_s)}{\|g_t\|_2 \|g_s\|_2} = \biggl[\sech \biggl(\frac
{t-s}{2} \biggr)
\biggr]^{\alpha+1},
\]
so $\{Y_t^{(\alpha)}, t \in\mathbb{R}\}$ is stationary and of the
specified nonnegative
covariance function. Next, since
\[
\hat{g}_t(r):= \frac{g_t(r)}{\|g_t\|_2} = \frac{r^{\alpha/2}}{\|g_0\|_2} \exp
\bigl(-t (\alpha+1)/2 - e^{-t} r\bigr),
\]
is infinitely differentiable in $t$ with
$\|\frac{d^k \hat{g}_t}{d t^k} \|_2$ finite
for all $k \in\mathbb{N}$, the sample functions
$t \mapsto Y_t^{(\alpha)} = \int_0^\infty\hat{g}_t(r) \,dW_r$
of (\ref{Ydef}) are $\mathcal{C}^\infty(\mathbb{R})$-valued.

The limit (\ref{bk}) for $\delta_T \equiv0$ is
merely $b(F^{\alpha+1})$ for covariance function
\mbox{$F^{\alpha+1} \in\mathcal{S}_+$}. Further, with
$\tau\mapsto\rho_\alpha(\tau):=[\sech(\tau/2)]^{\alpha+1}$
decreasing and satisfying the condition of \cite{LS2}, Remark 3.1,
it follows from \cite{LS2}, Theorem 3.1(iii), that
(\ref{bk}) extends to any $\delta_T \to0$.

By yet another application of Slepian's inequality, the
stated monotonicity properties of $\alpha\mapsto b_\alpha$
are immediate consequence of the monotonicity of
$\alpha\mapsto\rho_\alpha(\tau/(\alpha+1))$
and $\alpha\mapsto\rho_\alpha(\tau/\sqrt{\alpha+1})$,
per fixed $\tau$. Applying the monotone transformation
$-\log(\cdot)$ to these two functions of $\alpha+1$ and
setting $f(u):=\log\cosh(u)$, the preceding is in turn
equivalent to $u \mapsto u^{-1} f(u)$
nondecreasing and $u \mapsto u^{-2} f(u)$ nonincreasing
on $(0,\infty)$. The former holds since
\[
\psi_1 (u):= u^2 \bigl(u^{-1} f(u)
\bigr)' = u f'(u) - f(u)
\]
is such that $\psi_1'(u) = u f''(u) = u \sech^2(u) \ge0$,
hence $u \mapsto\psi_1(u)$ is nondecreasing, starting at
$\psi_1(0)=-f(0)=0$. So, necessarily both $\psi_1(u)$ and
$u^{-2} \psi_1(u) = (u^{-1} f(u))'$ are nonnegative for
$u>0$, from which it follows that\vadjust{\goodbreak} $u^{-1} f(u)$ is
nondecreasing. Similarly, setting
\[
\psi_2 (u):= u^3 \bigl(u^{-2} f(u)
\bigr)' = u f'(u) - 2 f(u)
\]
and noting that $f'(0)=\tanh(0)=0$, results with
\[
\psi_2'(u) = u f''(u) -
f'(u) = \int_0^u \bigl(
f''(u) - f''(r) \bigr) \,dr
\le0,
\]
due to the monotonicity of $f''(u)=\sech^2(u)$. So,
with $u \mapsto\psi_2(u)$ nonincreasing on $(0,\infty)$ and
starting at $\psi_2 (0)= - 2 f (0) =0$, we deduce that
$\psi_2(u) \le0$, and hence also $u^{-3} \psi_2(u) = (u^{-2} f(u))'
\le0$,
as claimed.

With $u^{-1} f(u) \uparrow1$ as $u \uparrow\infty$,
when $\alpha\downarrow-1$ the autocorrelation
$\widetilde{A}_\alpha(0,\tau):= \rho_\alpha(|\tau|/(\alpha+1))$ of
$Y^{(\alpha)}_{t/(\alpha+1)}$ converges downward to the autocorrelation
function $\widetilde{A}_{-1} (0,\tau):= \exp(-|\tau|/2)$ of the
standard, stationary Ornstein--Uhlenbeck process $\{X_t,t\geq0\}$,
whose persistence exponent is $1/2$ (cf. \cite{DPSZ}, Lemma 2.5).
In view of~(\ref{bk}) and Slepian's inequality, this results with
\[
(\alpha+1)^{-1} b_\alpha= b(\widetilde{A}_\alpha)
\le b(\widetilde{A}_{-1}) = 1/2,
\]
whereas the convergence of $b(\widetilde{A}_\alpha)$ to
$b(\widetilde{A}_{-1})$
is established by applying Theorem~\ref{general}, as
in (\ref{cnc1}). Indeed, condition (\ref{c2}) of the
theorem holds since
$\widetilde{A}_\alpha(0,\tau) \le\widetilde{A}_0(0,\tau)=\rho
_0(\tau)$
decays exponentially in $\tau$, uniformly in $\alpha\le0$,
while by Lemma~\ref{polyhelp}, condition (\ref{c1}) holds
for all $z \in\mathbb{R}$ since in this setting
$p_{\alpha}^2 (u) = 2 (1-\widetilde{A}_{\alpha} (0,u)) \le2
(1-e^{-u/2}) \le u$
satisfies (\ref{suff2}), and the limiting Ornstein--Uhlenbeck
process $\{X_t,t \ge0\}$ of continuous sample path satisfies
condition (\ref{c3}) since, for example, it satisfies
(\ref{suff1}) by \cite{LS2}, Remark 3.1.

Similarly, since $u^{-2} f(u) \uparrow1/2$ for $u \downarrow0$,
the correlation functions
$\widehat{A}_\alpha(0,\tau):= \rho_\alpha(|\tau|/\sqrt{\alpha+1})$
of $Y^{(\alpha)}_{t/\sqrt{\alpha+1}}$, $\alpha> -1$,
converge downward to $\widehat{A}_\infty(0,\tau):=\break \exp(-\tau^2/8)$
when $\alpha\uparrow\infty$. Consequently, $\widehat{A}_\infty\in
\mathcal{S}_+$
is the covariance function of some centered, stationary Gaussian process
$\{\widehat{Z}_t,t\geq0\}$, having nonnegative
persistence exponent $\hat{b}_\infty:= b(\widehat{A}_\infty)$.
By Slepian's inequality and (\ref{bk}),
\[
(\alpha+1)^{-1/2} b_\alpha= b(\widehat{A}_\alpha) \le
b(\widehat {A}_\infty) = \hat{b}_\infty
\]
and $b(\widehat{A}_\alpha) \to b(\widehat{A}_\infty)$ as a
consequence of
applying Theorem~\ref{general} for $\widehat{A}_\alpha\in\mathcal
{S}_+$. Indeed,
in this setting we have the uniform (over $\alpha\ge0$), exponential
decay of $\widehat{A}_\alpha(0,\tau) \le\rho_0(\tau)$,
condition (\ref{suff2}) of Lemma~\ref{polyhelp} holds as
$p_{\alpha}^2 (u) = 2 (1-\widehat{A}_{\alpha} (0,u)) \le2
(1-e^{-u^2/8}) \le u^2/4$
and we dealt already in Remark~\ref{rmk-binom} with condition (\ref{suff1}),
and thereby (\ref{c3}).
Finally, noting that $\exp(-|\tau|/8) \leq\exp(-\tau^2/8)$
for $|\tau| \le1$ and applying Slepian's inequality twice, we find
that for all $T$,
\[
\mathbb{P}\Bigl(\sup_{t\in[0,T]}\widehat{Z}_t\leq0
\Bigr) \geq\mathbb {P}\Bigl(\sup_{t\in[0,1]} \widehat{Z}_t
\leq0\Bigr)^{\lceil T\rceil} \geq\mathbb{P}\Bigl(\sup_{t\in[0,1]}
X_{t/4}\leq0\Bigr)^{\lceil T\rceil}.
\]
Clearly, $\mathbb{P}(\sup_{t\in[0,1/4]} X_t\leq0)>0$, hence
$\hat{b}_\infty$ is finite.

\section{Proof of Theorem \texorpdfstring{\protect\ref{poly}}{1.3}}\label{sec-3}

\subsection{Asymptotics for $p_{[0,1]}(n)$ and \texorpdfstring{$p_{(1,\infty)}(n)$}{$p_{(1,infty)}(n)$}}

We start by stating the three lemmas used in proving
part (a) of Theorem~\ref{poly} (deferring their
proofs to Section~\ref{sec-4}).
First, due to smoothness of $Q_n(\cdot)$, for $\delta>0$ small,
$\operatorname{sgn}\{Q_n(e^{-u})\}$ is controlled by the value of $Q_n(1)$
when $|u| \le n^{-(1-\delta)}$ and by the values of $a_0$ or $a_n$ when
$|u| \ge n^{-\delta}$. Hence, as our next lemma states,
the contribution of this range of arguments to
persistence exponents is negligible.

%
\begin{lem}\label{ignore}
In the setting of Theorem~\ref{poly}:
\begin{longlist}[(a)]
\item[(a)] For any $\alpha\in\mathbb{R}$ and slowly varying $L(\cdot)$,
%
%
\begin{eqnarray}
\label{eqlast} \lim_{\delta\rightarrow0}\liminf _{n\rightarrow\infty} \frac{1}{T_n}\log\mathbb{P}\Bigl(
\sup_{|u| \le n^{-(1-\delta)}} \bigl\{ Q_n\bigl(e^{-u}\bigr)
\bigr\} < 0\Bigr)&=&0,
\\
\label{eqfirst-a} \lim_{\delta\rightarrow0}\liminf_{n\rightarrow
\infty} \frac{1}{T_n}\log\mathbb{P}\bigl(
Q_n\bigl(e^{-u}\bigr) < 0,\ \forall|u|\ge
n^{-\delta} \bigr)&=&0.
\end{eqnarray}

\item[(b)] If $\sum_{i} L(i) i^{\alpha}$ converges then $n \mapsto p_{[0,1]}(n)$
is bounded away from zero. More generally, if $\alpha\leq-1$ then
%
%
\begin{equation}
\label{eqfirst-b} \lim_{n\rightarrow\infty} \frac{1}{T_n}\log\mathbb{P}\Bigl(\sup
_{u \ge0}
\bigl\{ Q_n\bigl(e^{-u}\bigr) \bigr\} < 0\Bigr)=0.
\end{equation}
\end{longlist}
\end{lem}

Hereafter, for positive functions $f,g$ of common domain,
$f(x)\lesssim g(x)$ stands for existence of finite
uniform bound $\sup_x f(x)/g(x) \le C(\alpha,L(\cdot))$.

From (\ref{eqfirst-b}), we have
that $p_{[0,1]}(n) = n^{-o(1)}$ when $\alpha\le-1$,
and our next lemma is key to finding the contribution of
$u\in(n^{-(1-\delta)},n^{-\delta})$ to the
asymptotics of $p_{[0,1]}(n)$, in case $\alpha>-1$.

%
\begin{lem}\label{third}
For any $\alpha>-1$, $\delta>0$, slowly varying $L(\cdot)$ and
$h_{\alpha,n}(\cdot)$ as in~(\ref{defhaln}),
%
%
\begin{equation}
\label{h2} \lim_{n \to\infty} \sup_{w \in(2 n^{-(1-\delta)},2 n^{-\delta})} \biggl|
\frac{w^{\alpha+1} h_{\alpha,n}(w)}{L(1/w)} - \Gamma(\alpha+1) \biggr| = 0.
\end{equation}
Consequently, in the setting of Theorem~\ref{poly}, for $u,v \in
(n^{-(1-\delta)},n^{-\delta})$,
%
%
\begin{eqnarray}
\label{h5} \bar{c}_n(u,v) &:=&\operatorname{corr}\bigl[Q_n
\bigl(e^{-u}\bigr),Q_n\bigl(e^{-v}\bigr)\bigr]
\lesssim e^{-(\alpha+1)/4|\log v-\log u|}
\end{eqnarray}
and for any $M$ finite there exist $\epsilon_n = \epsilon_n(M)
\downarrow0$
such that if in addition $u/v \in[1/M,M]$, then
%
%
\begin{eqnarray}\label{h4}
(1-\epsilon_n) R(u,v)^{\alpha+1}+
\epsilon_n R(u,v)^{\alpha+2} &\le& \bar{c}_n(u,v)
\nonumber\\
&\le& (1-\epsilon_n) R(u,v)^{\alpha+1}+\epsilon_n
\end{eqnarray}
[for $R(\cdot,\cdot)$ of (\ref{deflimitcov})].
\end{lem}

Similarly, the following lemma controls the
contribution of $x \in(e^{n^{-(1-\delta)}},\break e^{n^{-\delta}})$
to $p_{(1,\infty)}(n)$.

%
\begin{lem}\label{after}
For $h_{\alpha,n}(\cdot)$ of (\ref{defhaln}), any $\alpha\in
\mathbb{R}$,
$\delta>0$ and slowly varying
$L(\cdot)$,
as $n \to\infty$,
%
%
\begin{equation}
\label{hh1} \sup_{w \in(2 n^{-(1-\delta)}, 2 n^{-\delta})} \biggl| \frac{w e^{-n w} h_{\alpha,n}(-w)}{L(n) n^{\alpha}} - 1 \biggr| \to 0.
\end{equation}
Consequently, for all $u,v \in(n^{-(1-\delta)},n^{-\delta})$,
%
%
\begin{eqnarray}
\label{hh3} \tilde{c}_n(u,v):=\operatorname{corr}\bigl[Q_n
\bigl(e^{u}\bigr),Q_n\bigl(e^{v}\bigr)\bigr]
&\lesssim&e^{-1/2|\log v-\log u|}
\end{eqnarray}
and for any $M$ finite there exist $\epsilon_n = \epsilon_n(M)
\downarrow0$
such that if in addition $u/v \in[1/M,M]$, then
%
%
\begin{equation}
\label{hh2} (1-\epsilon_n) R (u,v)+\epsilon_n
R(u,v)^2\le\tilde{c}_n(u,v)\le (1-
\epsilon_n) R (u,v)+\epsilon_n.
\end{equation}
\end{lem}

\begin{pf*}{Proof of part (\textup{a}) of Theorem~\ref{poly}}
Starting with the proof of (\ref{ulb1}), we fix $\delta>0$ and
partition $\mathbb{R}_+$ into three disjoint intervals
$\overline{J}_H=[n^{-\delta},\infty)$, $\overline{J}=(n^{-(1-\delta
)},n^{-\delta})$
and $\overline{J}_L=[0,n^{-(1-\delta)}]$. Then,
with $\overline{Q}_n(u):=Q_n(e^{-u})/\sqrt{h_{\alpha,n}(2u)}$,
by Slepian's
inequality and the nonnegativity of the covariance of $Q_n(\cdot)$,
we have that
\begin{eqnarray*}
&& \mathbb{P}\Bigl(\sup_{u\in\overline{J}}\bigl\{ \overline{Q}_n(u)
\bigr\}<0\Bigr)
\\
&&\qquad \ge \mathbb{P}\Bigl(\sup_{x\in[0,1]}\bigl
\{Q_n(x)\bigr\}<0\Bigr)
\\
&&\qquad \ge \mathbb{P}\Bigl(\sup_{u\in\overline{J}}\bigl\{ \overline{Q}_n(u)
\bigr\}<0\Bigr) \mathbb{P}\Bigl(\sup_{u\in\overline{J}_L}\bigl\{
\overline{Q}_n(u)\bigr\}<0\Bigr) \mathbb{P}\Bigl(\sup
_{u\in\overline{J}_H}\bigl\{\overline{Q}_n(u)\bigr\}<0\Bigr).
\end{eqnarray*}
Considering the limit of $\frac{1}{T_n}\log(\cdot)$ of these probabilities
as $n\rightarrow\infty$ followed by $\delta\downarrow0$, we have by
Lemma~\ref{ignore} that suffices to consider $\alpha>-1$, and only
the term
involving $u \in\overline{J}$ is relevant for the asymptotics of $p_{[0,1]}(n)$.
To deal with the latter term, let
\[
A_n(s,t):=\bar{c}_n\bigl(\exp\bigl
\{-e^{-s}/n^\delta\bigr\},\exp\bigl\{ -e^{-t}/n^\delta
\bigr\}\bigr)
\]
so that $u,v \in\overline{J}$ correspond to $s:=-\log u-\delta T_n$ and
$t:=-\log v-\delta T_n$, in $[0,(1-2\delta) T_n]$. Upon this change
of variables, the inequalities (\ref{h4}) of Lemma~\ref{third}
translates into (\ref{c11}) holding for $A_\infty(s,t):=
F(s,t)^{\alpha+1}$
and $D(s,t):=F(s,t)^{\alpha+2}$ in $\mathcal{S}_+$, the covariance
functions of
processes $Y_t^{(\alpha)}$ and $Y_t^{(\alpha+1)}$ of continuous
sample path.
Hence, by Lemma~\ref{polyhelp} condition (\ref{c1}) of Theorem~\ref{general}
holds, whereas by (\ref{bk}) of Lemma~\ref{0} so does condition (\ref{c3}),
and from (\ref{h5}) we have that $A_n(s,t) \leq C \exp(-\frac{\alpha
+1}{4} |t-s|)$
for some $C$ finite, any $n$ and all \mbox{$s,t \in[0,(1-2\delta)T_n]$},
which is much stronger than condition (\ref{c2}). We thus conclude from
Theorem~\ref{general} (for $T=T_n \to\infty$, as in Remark~\ref{rmk-Tk}), that
%
%
\begin{equation}
\label{J3bd} \lim_{n\rightarrow\infty}\frac{1}{T_n}\log\mathbb{P}\Bigl(\sup
_{u\in
\overline{J}} \bigl\{ \overline{Q}_n (u) \bigr\} < 0\Bigr)
= - (1-2\delta) b_\alpha
\end{equation}
from which (\ref{ulb1}) follows upon taking $\delta\downarrow0$.

Similarly, for proving (\ref{ulb2}) we fix $\delta>0$ and considering
$\widehat{Q}_n(w):= Q_n(e^w)/\break \sqrt{h_{\alpha,n}(-2w)}$, split the
supremum over $w \in\mathbb{R}_+$ into the disjoint $\overline{J}_L$,
$\overline{J}$ and $\overline{J}_H$, of which by Lemma~\ref{ignore} only
the supremum over $w \in\overline{J}$ matters. Same change of variable
yields covariance functions
$A_n(s,t):= \tilde{c}_n(\exp\{-e^{-s}/n^\delta\},\break \exp\{
-e^{-t}/n^\delta\})$
for $s,t \in[0,(1-2\delta) T_n]$, which in view of (\ref{hh2}) of
Lemma~\ref{after} satisfy (\ref{c11}) for $A_\infty(s,t)=F(s,t)$
and $D(s,t)=F(s,t)^2$, whereas the bound (\ref{hh3}) of that lemma
provides uniform exponential decay $A_n(s,t) \leq\break C \exp(-|t-s|/2)$.
Put together, by yet another application of Lemmas~\ref{polyhelp} and \ref{0}, and Theorem~\ref{general}, we conclude that
%
%
\begin{equation}
\label{J4bd} \lim_{n\rightarrow\infty}\frac{1}{T_n}\log\mathbb{P}\Bigl(\sup
_{u
\in\overline{J}} \bigl\{ \widehat{Q}_n(u) \bigr\} < 0
\Bigr)=-(1-2\delta) b_0,
\end{equation}
so letting $\delta\downarrow0$ we arrive at (\ref{ulb2}).

Turning to prove (\ref{p3}), since $Q_n(x)$ has nonnegative
correlation on $[0,\infty)$, by Slepian's inequality, for any
slowly varying $L(\cdot)$ and all $n$, the lower bound
%
%
\begin{equation}
\label{eqlbd-rplus} p_{[0,\infty)}(n) \ge n^{-b_\alpha-b_0 - o(1)}
\end{equation}
as in (\ref{p3}), is a direct consequence of the
corresponding lower bounds of (\ref{ulb2})~and~(\ref{ulb1}),
and the matching upper bound for (\ref{p3}) is derived
in the sequel [while upper bounding $p_{\mathbb{R}}(n)$].
\end{pf*}

\subsection{Lower bound on $p_{\mathbb{R}}(n)$}
Having centered Gaussian coefficients, the joint law of
$\{Q_n(x)\dvtx  x \in\mathbb{R}\}$ is invariant under $x \mapsto-x$,
hence same lower bound applies for $p_{(-\infty,0]}(n)$.
Consequently, for the stated lower bound on
$p_{\mathbb{R}}(2n)$, it suffices to establish
strong control on
$\operatorname{corr}[{Q}_n(x),{Q}_n(-y)]$ for $x,y>0$.

Unfortunately, in case $x=y \in(0,1)$ fixed,
these correlations \textit{do not} decay with~$n$. However, the nonnegligible correlation
comes from lower order coefficients of
$Q_n(\cdot)$, so our first order of business
is to show that suffices to consider only the
higher order part of $Q_n(\cdot)$.

Indeed, by definition, for any slowly varying $L(\cdot)$
there exists
$r \in\mathbb{N}$ such that $L(i) > 0$ for all $i \ge2r$.
Further, as $\rho\downarrow0$, uniformly in $|x| \le1$
\[
f_\rho(x):= 1 + x^{2r} - \rho\sum
_{i=1}^{r} |x|^{2i-1} \to
f_0(x) \ge 1
\]
and $f_\rho(x)$ is nondecreasing
in $|x| \ge1$ for all $\rho$ small enough, hence\break
$\inf_x f_{\rho_0} (x) > 0$ for some $\rho_0 > 0$.
Fixing $\delta>0$, set
$m=m_n:= \lceil\delta T_n \rceil$ and with
$\hat{a}_i$ denoting independent centered Gaussian variables of
variances $(3/4) \mathbb{E}[a_i^2]$, independent of the sequence $\{
a_i\}$,
note that $Q_n(\cdot)=Q_n^{L}(\cdot) + Q_n^{M}(\cdot) + Q_n^H(\cdot)$,
for the independent algebraic polynomials,
\begin{eqnarray*}
{Q}_n^L(x)&:=& \hat{a}_0 + \sum
_{i=1}^{2r-1} a_i x^i
+ \hat{a}_{2 r} x^{2r},
\\
{Q}_n^M(x)&:=& 0.5 \sum_{i=r}^{m-1}
x^{2i} \bigl[a_{2i} + 2 a_{2i+1} x +
a_{2i+2} x^2\bigr], 
\\
{Q}_n^H(x)&:=& 0.5 a_0 +
\hat{a}_{2 m} x^{2 m} + \sum_{i=2 m +
1}^{n}
a_i x^i.
\end{eqnarray*}
For any $\rho>0$, the event
\[
\Gamma_\rho:= \Bigl\{ \hat{a}_0 \le-1, \sup
_{i=1}^{r-1} \{ a_{2i}\} \le0,  \sup
_{i=1}^{r} \bigl\{|a_{2i-1}|\bigr\} \le\rho,
\hat{a}_{2r} \le- 1 \Bigr\},
\]
of positive probability [as $\mathbb{E}[a_0^2] L(2r) > 0$], results with
$Q_n^L(\cdot) \le- f_\rho(\cdot)$. Hence,
\[
\mathbb{P}\Bigl(\sup_{x\in\mathbb{R}} \bigl\{ Q^L_n(x)
\bigr\} <0\Bigr) \ge \mathbb{P} (\Gamma_{\rho_0}) > 0.
\]
Next, if $a_{2i} \le0$ and $a_{2i} a_{2i+2} \ge a_{2i+1}^2$ for
all $r \le i \le m-1$, then necessarily $Q^M_n(x) \le0$ for all
$x \in\mathbb{R}$. Due to strict positivity of the slowly varying $L(2i)$
for $i \ge r$,
\[
c_{2i}:= \frac{L(2i+1)}{\sqrt{L(2i) L(2i+2)}} \biggl( \frac{(2i+1)^2}{(2i)(2i+2)}
\biggr)^{\alpha/2}
\]
is uniformly bounded for $i \ge r$, for example,
$C:= \sup_{i\geq r} \{ c_{2i} \}$ is finite and
with $a_{i} = \sqrt{i^{\alpha} L(i)} Z_i$ for
standard i.i.d. Gaussian $\{Z_i\}$,
the preceding event occurs whenever
$Z_{2i} \le-\sqrt{C}$ and $|Z_{2i+1}| \le1$ for all
$r \le i \le m$. That is, for some positive
$\lambda=\lambda(C) < \mathbb{P} (\Gamma_{\rho_0})$
and all $n$ large
\[
\mathbb{P}\Bigl(\sup_{x\in\mathbb{R}} \bigl\{ Q^M_n(x)
\bigr\} \le0\Bigr) \ge\lambda ^{m}.
\]
By the preceding and independence of these three polynomials,
%
%
\begin{eqnarray}
\label{july1} \qquad p_{\mathbb{R}}(n)&\geq& \mathbb{P}\Bigl(\sup
_{x\in\mathbb{R}} \bigl\{Q_n^L(x) \bigr\} < 0,
\sup_{x\in\mathbb{R}} \bigl\{Q_n^M(x) \bigr\}
\le0, \sup_{x\in\mathbb{R}} \bigl\{Q_n^H(x) \bigr
\} \le0 \Bigr)
\nonumber\\
&\ge& \lambda^{m+1} \mathbb{P}\Bigl(\sup_{x\in\mathbb{R}} \bigl
\{ \widetilde{Q}_n (x) \bigr\} \leq0\Bigr),
\end{eqnarray}
where $\widetilde{Q}_n (x):=\frac{Q_n^H(x)}{\sqrt{\operatorname{var}(Q_n^H(x))}}$ and
$d_n(x,y):=\operatorname{corr} [{Q}_n^H(x),{Q}_n^H(y) ]$.
Note that the covariance of $Q^H_n(e^{-\cdot})$ is
$0.25 + h_{\alpha,n}(\cdot) - h_{\alpha,2m-1}(\cdot)$
and $m=m_n=O(\log n)$ is small enough that both
(\ref{h4}) and (\ref{hh2}) apply for $d_n(e^{-u},e^{-v})$.
It is further not hard to check that Lemma~\ref{ignore}
holds for $Q_n^H(\cdot)$. Thus, by
a rerun of the proof of part (a) of Theorem~\ref{poly}
we arrive at the analog of (\ref{eqlbd-rplus})
for $Q^H_n(\cdot)$. Namely, that if $\xi_n \to0$ as $n \to\infty$, then
%
%
\begin{equation}
\label{eqlbd-H} \mathbb{P}\Bigl(\sup_{x \ge0} \bigl\{
\widetilde{Q}_n (x) \bigr\} \le\xi_n\Bigr) \ge
n^{-b_\alpha- b_0 - o(1)}.
\end{equation}
We show in the sequel that subject to condition (\ref{rate3}) on
$L(\cdot)$,
for even values of $n \to\infty$,
%
%
\begin{equation}
\label{B1} \gamma_n:= - m_n \inf
_{xy>0} \bigl\{ d_n(x,-y) \wedge0 \bigr\} \to0.
\end{equation}
This implies that for $\epsilon_n = 2 \gamma_n/m_n$,
\[
(1-\epsilon_n)\,d_n(x,y)+\epsilon_n\geq
d_n(x,y) 1_{\{xy \ge0\}},
\]
hence with $\xi_n:= - \gamma_n^{1/4}$ [so $\xi_n^2/\epsilon_n = m_n
/(2 \sqrt{\gamma_n})$],
and $Z$ a standard Gaussian independent of $\widetilde{Q}_n(\cdot)$,
it follows from Slepian's inequality and the union bound that
\begin{eqnarray*}
\mathbb{P}\Bigl(\sup_{x\in\mathbb{R}} \bigl\{ \widetilde{Q}_n(x)
\bigr\} \le0\Bigr) &\geq& \mathbb{P}\Bigl( \sup_{x\in\mathbb{R}} \bigl\{
\sqrt{1-\epsilon_n} \widetilde {Q}_n(x)+\sqrt{
\epsilon_n} Z \bigr\} \le\xi_n\Bigr) - \mathbb{P}(\sqrt{
\epsilon_n} Z \le\xi_n)
\\
&\geq& \Bigl[ \mathbb{P}\Bigl(\sup_{x \ge0} \bigl\{
\widetilde{Q}_n (x) \bigr\} \le \xi _n\Bigr)
\Bigr]^2 - e^{-m_n/(4 \sqrt{\gamma_n})}.
\end{eqnarray*}
Considering $T_n^{-1}\log(\cdot)$ of both sides and taking $n \to
\infty
$ followed by $\delta\downarrow0$, we
conclude in view of (\ref{july1}), (\ref{eqlbd-H}) and our choice of
$m = m_n = \lceil\delta T_n \rceil$, that
\[
\liminf_{n\rightarrow\infty}\frac{1}{T_n}\log p_{\mathbb{R}}(n) \ge2 \lim
_{\delta\downarrow0} \liminf_{n\rightarrow\infty}\frac{1}{T_n}\log\mathbb {P}\Bigl(
\sup_{x
\ge0} \bigl\{ \widetilde{Q}_n (x) \bigr\} \le
\xi_n\Bigr) \ge-2 (b_\alpha+ b_0).
\]
Proceeding to prove (\ref{B1}), note that for $x,y \ge0$,
\[
d_n(x,-y)= d_n(x,y) \biggl[ \frac{0.25 + h^\delta_e(xy)-h^\delta_o(xy)}{0.25
+ h^\delta_e(xy) + h^\delta_o(xy)} \biggr],
\]
where, assuming hereafter that $n$ is an \textit{even} integer,
\begin{eqnarray*}
h^\delta_e(z)&:=& \sum_{i=m+1}^{n/2}
L(2i) (2i)^{\alpha} z^{2i} + \frac{3}{4} L(2m)
(2m)^{\alpha} z^{2m},
\\
h^\delta_o(z)&:=& \sum_{i=m+1}^{n/2}
L(2i-1) (2i-1)^{\alpha} z^{2i-1}.
\end{eqnarray*}
With $d_n(x,y) \in[0,1]$, we thus get (\ref{B1}) by showing that for
some $\gamma_n \to0$,
%
%
\begin{equation}
\label{B2} h^\delta_e(z) \ge\bigl(1 -
\gamma_n m_n^{-1}\bigr) h^\delta_o
(z)\qquad \forall z \ge0.
\end{equation}
To this end, setting $C_{2i-1}:= \sqrt{L(2i) L(2i-2) (2i)^{\alpha}
(2i-2)^{\alpha}}$,
observe that with $n$ even [and $L(\cdot)$ nonnegative], by
discriminant calculations
similar to those we used for bounding $Q_n^M(\cdot)$,
\[
h^\delta_e(z) \ge\sum_{i=m+1}^{n/2}
C_{2i-1} z^{2i-1}\qquad \forall z \in\mathbb{R}.
\]
Hence, (\ref{B2}) follows from
\[
\limsup_{i \to\infty} (2i-1) \biggl| \frac{C_{2i-1}}{L(2i-1)
(2i-1)^{\alpha}} - 1 \biggr| = 0,
\]
which for $\alpha$ finite is a direct consequence of our assumption
(\ref{rate3}).

\subsection{Upper bound on $p_{\mathbb{R}}(n)$}

Considering first the case of $\alpha> -1$, we fix $\delta>0$ and
have that
\[
p_{\mathbb{R}}(n)\leq\mathbb{P}\Bigl(\sup_{x\in I_n(\delta)} \bigl\{
Q_n(x) \bigr\} <0\Bigr),
\]
where
\[
I_n(\delta):= \pm \bigl\{\bigl(e^{-n^{-(1-\delta)}}, e^{-n^{-\delta}}
\bigr) \cup\bigl(e^{n^{-(1-\delta)}},e^{n^{-\delta}}\bigr) \bigr\} =:\bigcup
_{i=1}^4 J_i(\delta).
\]
The asymptotic of $p_{J_3(\delta)}(n)$ and $p_{J_4(\delta)}(n)$, provided
in (\ref{J3bd}), and (\ref{J4bd}), respectively, extend to any crossing
levels $\xi_n \to0$. In view of these and the
invariance of law of $Q_n(\cdot)$ to change of sign, by the
usual argument based on Slepian's inequality, it remains only to show
that the autocorrelation $c_n(x,y):=\operatorname{corr}[Q_n(x),Q_n(y)]$ satisfies
%
%
\begin{eqnarray}
\label{B3} c_n(x,y)&\leq&\epsilon_n+(1-
\epsilon_n)c_n(x,y)1_{\{(x,y)\in
J_i(\delta
),1\leq i\leq4\}}
\end{eqnarray}
for some $\epsilon_n T_n \to0$. This amounts to confirming that
%
%
\begin{eqnarray}
\label{y1} T_n c_n(x,-y) &\lesssim&o(1)\qquad\forall x,y \in \bigl(e^{-n^{-\delta}},e^{n^{-\delta}}\bigr),
\\
\label{y2} T_n c_n\bigl(x,y^{-1}\bigr)
&\lesssim&o(1)\qquad\forall x,y \in \bigl(e^{-n^{-\delta}},e^{-n^{-(1-\delta)}}
\bigr).
\end{eqnarray}
Turning to prove (\ref{y1}), note that
\[
\operatorname{Cov}\bigl(Q_n(x),Q_n(y)\bigr)=h_e(xy)+h_o(xy)
\]
for
\[
h_e(z):=1+\sum_{i=1}^{n/2} L(2i)
(2i)^{\alpha}z^{2i},\qquad h_o(z):=\sum
_{i=1}^{n/2} L(2i-1) (2i-1)^{\alpha}
z^{2i-1}.
\]
Thus,
\[
\bigl|c_n(x,-y)\bigr|=c_n(x,y)\frac{|h_e(xy)-h_o(xy)|}{h_e(xy)+h_o(xy)}\leq
\frac
{|h_e(xy)-h_o(xy)|}{h_e(xy)+h_o(xy)}
\]
and it suffices to show that as $n \to\infty$,
%
%
\begin{equation}
\label{y4} T_n \sup_{|\log z| \le2 n^{-\delta}} \frac
{|h_e(z)-h_o(z)|}{h_e(z)+h_o(z)}
\to0.
\end{equation}
To this end, setting $m = m_n:=\lfloor T_n^2\rfloor$ we have by (\ref{rate3}) that
\begin{eqnarray*}
&& \bigl|h_e(z)-h_o(z)\bigr|
\\
&&\qquad \leq  1+\sum_{i=1}^{2m}
L(i) i^{\alpha} z^i +\sum_{i=m+1}^{n/2}
L(2i) (2i)^{\alpha} z^{2i} \biggl|\frac
{L(2i-1)(2i-1)^{\alpha}}{L(2i)(2i)^{\alpha}}z^{-1}-1\biggr|
\\
&&\qquad \lesssim \sum_{i=1}^{2m} i^{\alpha+\delta}+\sum
_{i=m+1}^{n/2} \biggl[\biggl|1-\frac{1}{z}\biggr|+\sup
_{i\geq m} \biggl|\frac{L(2i-1)(2i-1)^{\alpha}}{L(2i)(2i)^{\alpha}}-1 \biggr| \biggr]L(2i)
(2i)^{\alpha}z^{2i}
\\
&&\qquad \lesssim  T_n^{2(\alpha+2)_+} + \bigl[n^{-\delta}+m_n^{-1}\bigr]h_e(z).
\end{eqnarray*}
Noting that $z \mapsto[h_e(z) + h_o(z)]$ is nondecreasing on $\mathbb{R}_+$,
we get from
(\ref{h2}) that
\[
\inf_{|\log z| \le2 n^{-\delta}} \bigl[h_e(z)+h_o(z)
\bigr] \gtrsim L\bigl(n^\delta \bigr) n^{\delta(\alpha+1)} \gtrsim
n^{\delta(\alpha+1)/2}
\]
and (\ref{y4}) follows. Proceeding to prove (\ref{y2}),
note that $\max(x,y)^n \le e^{-n^\delta}$ for $x,y\in J_3(\delta)$, hence
\begin{eqnarray*}
c_n\bigl(x,y^{-1}\bigr)&=&\frac{y^{n}+\sum_{i=1}^{n}L(i) i^{\alpha} x^i
y^{n-i}}{ [ (1+\sum_{i=1}^{n}L(i) i^{\alpha} x^{2i})
(y^{2n} +\sum_{i=1}^{n}L(i) i^{\alpha} y^{2(n-i)})
]^{1/2} }
\\
& \lesssim&
\frac{n^{\alpha+2}\max(x,y)^n}{\sqrt{L(n) n^{\alpha}}}\lesssim e^{-n^{\delta/2}}.
\end{eqnarray*}

Finally, in case $\alpha\le-1$ it suffices to consider the event of
no-crossing in intervals $J_1(\delta) \cup J_4(\delta)$ outside
$[-1,1]$. Consequently, suffices to confirm only (\ref{y1}), the first
of our two claims, and only for $x,y \in J_4(\delta):=
(e^{n^{-(1-\delta)}},e^{n^{-\delta}})$.
We proceed as before via (\ref{y4}), now needing it only for $\sqrt{z}
\in J_4(\delta)$, so at end of its proof we rely
here on the bound (\ref{hh1}) at $w=2 n^{-(1-\delta)}$ (which hold for
all $\alpha\in\mathbb{R}$), to get that
uniformly in $\sqrt{z}\in J_4(\delta)$,
\[
h_e(z)+h_o(z)\gtrsim n^{1-\delta} L(n)
n^{\alpha} e^{2 n^\delta} \gtrsim e^{n^{\delta}}.
\]


\section{Proofs of Lemmas \texorpdfstring{\protect\ref{ignore}}{3.1}--\texorpdfstring{\protect\ref{after}}{3.3}}\label{sec-4}

We begin by proving Lemmas~\ref{third} and~\ref{after}
regarding asymptotic covariances in intervals which
dominate the persistence probabilities of
Theorem~\ref{poly}.
\begin{pf*}{{ Proof of Lemma~\ref{third}}}
We set $\overline{J}:=(n^{-(1-\delta)},n^{-\delta})$ and
make frequent use of the following obvious estimates, valid for all
$l>-1$ and $y>1>w>0$:
\begin{eqnarray*}
w^{l+1} \sum_{i \ge y/w} i^{l}
e^{-iw} &\lesssim& e^{-y/2},\qquad w^{l+1} \int
_{x \ge y/w} x^{l}e^{-x w}\,dx\lesssim
e^{-y/2},
\\
w^{l+1} \sum_{i=1}^{1/w} i^{l}&\lesssim & 1.
\end{eqnarray*}
Here, the constants implied by $\lesssim$ are allowed to depend on $l$
(in any case we use these bounds only for $l=\alpha$,
$l=\alpha+1$ and $l=\alpha+2$).

Starting with the proof of (\ref{h2}), from the representation theorem
\cite{BGT}, Theorem~1.3.1,
it follows that  $L(x) \sim \tilde{L}(x)$ and $x^\eta \tilde{L}(x)$ is
eventually increasing (decreasing), if $\eta>0$ (or
$\eta<0$, resp.).
Hence, to simplify the presentation we can assume
hereafter that $x^\eta L(x)$ is eventually increasing (decreasing) if
$\eta>0$ (or $\eta<0$, resp.).
Thus, for $\eta:=(l+1)/2>0$ there exists $x_1<\infty$ such that
$L(i)\leq L(1/w)/(wi)^{\eta} $ for all $ x_1\leq i\leq1/w$.
Consequently, for all $a \ge w x_1$,
%
%
\begin{equation}
\label{t1} \frac{w^{l+1}}{L(1/w)} \sum_{i=x_1}^{a/w}
L(i) i^l e^{-i w} \leq w^{l+1-\eta} \sum
_{i=x_1}^{a/w} i^{l-\eta}e^{-iw}
\lesssim a^{(l+1)/2}.
\end{equation}
Likewise, there exists $x_2<\infty$ such that $L(i)\leq i w L(1/w)$ for
$x_2\leq1/w \leq i$; hence, for $b\geq w x_2$,
%
%
\begin{equation}
\label{t2} \frac{w^{l+1}}{L(1/w)} \sum_{i\geq b/w}
L(i)i^le^{-iw} \leq w^{l+2} \sum
_{i\geq
b/w} i^{l+1} e^{-iw} \lesssim e^{-b/2}.
\end{equation}
Combining the bounds (\ref{t1}) and (\ref{t2}) with those corresponding
to $L(\cdot)\equiv1$, results with
\begin{eqnarray*}
&& \frac{w^{l+1}}{L(1/w)} \Biggl|\sum_{i=x_1}^\infty \biggl[L(i)-L
\biggl(\frac{1}{w}\biggr)\biggr]i^le^{-iw} \Biggr|
\\
&&\qquad \lesssim
a^{(l+1)/2}+e^{-b/2}+ \biggl\{\sup_{\lambda\in[a,b]} \biggl|
\frac{L(\lambda
/w)}{L(1/w)}-1 \biggr| \biggr\} w^{l+1}\sum_{i=x_1}^\infty
i^l e^{-iw}.
\end{eqnarray*}
Since for $l + 1 >0$ and $w >0$,
\[
\Biggl| w^{l+1} \sum_{i=x_1}^\infty
i^l e^{-iw} - \Gamma(l+1) \Biggr| \lesssim w^{\min(l+1,1)},
\]
it follows that for any $n \ge b/w$,
%
%
\begin{eqnarray}
\label{estimate}
&&\biggl|\frac{w^{l+1}h_{l,n}(w)}{L(1/w)}- \Gamma(l+1) \biggr|
\nonumber\\
&&\qquad \lesssim a^{(l+1)/2}+e^{-b/2}
+\sup_{\lambda\in[a,b]} \biggl| \frac{L(\lambda/w)}{L(1/w)}-1 \biggr|
+ w^{\min(l+1,1)/2}.
\end{eqnarray}
To deduce (\ref{h2}), consider $l=\alpha>-1$ and fixing $\epsilon>0$,
choose $a=a(\epsilon)$ small and $b=b(\epsilon)$ large such that
for all $w \in2 \overline{J}$ the first two terms on the right-hand side
are bounded by $\epsilon$. Then recall that for $w \downarrow0$, the
convergence
$|L(\lambda/w)/L(1/w)-1| \to0$ is uniform over $\lambda$ in compacts
(cf. \cite{BGT}, Theorem 1.2.1).

Turning to prove (\ref{h5}), we have by (\ref{h2}) that for $u,v \in
\overline{J}$,
\[
\bar{c}_n(u,v)=\frac{h_{\alpha,n}(u+v)}{\sqrt{h_{\alpha,n}(2u)h_{\alpha,n}(2v)}} \lesssim S(u,v)
R(u,v)^{\alpha+1}
\]
with $S(\cdot,\cdot)$ and $R(\cdot,\cdot)$ of (\ref{deflimitcov}). By
the eventual monotonicity of
$x\mapsto x^{\pm2\eta} L(x)$, we further have for $n^{-\delta} \ge v
\ge u >0$ and all large $n$,
\[
\sqrt{\frac{L(1/(u+v))}{L(1/(2u))}}\leq \biggl(\frac
{u+v}{2u} \biggr)^\eta,
\qquad \sqrt{ \frac{L(1/(u+v))}{L(1/(2v))}}\leq \biggl(\frac
{2v}{u+v}
\biggr)^\eta,
\]
resulting with $S(u,v) \le(v/u)^{\eta}$. Clearly, $R(u,v) \le2
(v/u)^{-1/2}$, so taking
$\eta=(\alpha+1)/4$ we arrive at (\ref{h5}).
Next, fixing $M>1$ and setting
$\bar{g}_{\alpha,n}(w):=w^{\alpha+1}h_{\alpha,n}(w)$,
\[
\overline{G}_{\alpha,n}(u,v):= \frac{\bar{c}_n(u,v)}{R(u,v)^{\alpha
+1}} = \frac{\bar{g}_{\alpha,n}(u+v)}{\sqrt{\bar{g}_{\alpha,n}(2u)\bar{g}_{\alpha,n}(2v)}}
\]
[by (\ref{deflimitcov}) and the preceding expression for $\bar{c}_n(u,v)$], our claim (\ref{h4}) amounts to
%
%
\begin{equation}
\label{nn1} - \epsilon_n \bigl(1 - R(u,v)\bigr)\le
\overline{G}_{\alpha,n}(u,v) - 1 \le \epsilon_n
\bigl(R(u,v)^{-(\alpha+1)} - 1\bigr)
\end{equation}
for some $\epsilon_n \to0$, any $v \in[u,Mu]$ and all $u \in\overline{J}$.
Since $z - 1 - \log z \ge0$ on $\mathbb{R}_+$ and
$\epsilon p (1-r) \le\log(1 + \epsilon(r^{-p} - 1))$ whenever
$p \ge0$ and $r,\epsilon\in[0,1]$, the inequality~(\ref{nn1})
follows in turn from
\[
- \epsilon_n \bigl(1 - R(u,v)\bigr)\le G_{\alpha,n}(u,v):=
\log\overline {G}_{\alpha,n} (u,v) \le\epsilon_n (\alpha+1)
\bigl(1-R(u,v)\bigr).
\]
To this end, setting $\epsilon_n:= (1+\alpha\wedge0)^{-1} (1+M)^2
\tilde{\epsilon}_n$ and noting that
\[
1-R(u,v)=\frac{(\sqrt{v}-\sqrt{u})^2}{v+u} \geq\frac{(v-u)^2}{2
(v+u)^2} \ge\frac{(v-u)^2}{2 (1+M)^2 u^2},
\]
it suffices to show that for some $\tilde{\epsilon}_n \to0$,
%
%
\begin{equation}
\label{n3} \bigl|G_{\alpha,n}(u,v)\bigr| \leq\tilde{\epsilon}_n
\frac{(v-u)^2}{2
u^2}.
\end{equation}
Now, fixing $u$, we expand the function $v \mapsto G_{\alpha,n}(u,v)$
in Taylor's series about $v=u$, to get
%
%
\begin{equation}
\label{now2} \qquad G_{\alpha,n}(u,v)=G_{\alpha,n}(u,u)+(v-u){G}_{\alpha,n}'(u,u)+
\frac
{(v-u)^2}{2}{G}_{\alpha,n}''(u,\xi)
\end{equation}
for some $\xi=\xi_n(u,v) \in[u,v]$. With
\[
G_{\alpha,n}(u,v) = g_{\alpha,n}(u+v)-\tfrac{1}{2}g_{\alpha,n}(2u)-
\tfrac
{1}{2}g_{\alpha,n}(2v), \qquad g_{\alpha,n}(w):=\log
\bar{g}_{\alpha,n}(w),
\]
clearly $G_{\alpha,n}(u,u)=G_{\alpha,n}'(u,u)=0$ and
%
%
\begin{eqnarray}
\label{der}
u^2 \bigl|G_{\alpha,n}''(u,
\xi)\bigr|&=& u^2 \bigl|g_{\alpha,n}''(u+\xi)- 2
g_{\alpha,n}''(2\xi)\bigr|
\nonumber\\
&\leq& 3 \sup_{w \in2\overline{J}} \bigl\{ w^2 \bigl|g_{\alpha,n}''(w)\bigr|
\bigr\}:= \tilde{\epsilon}_n.
\end{eqnarray}
Thus, to complete the proof of (\ref{n3}), and thereby that of (\ref{h4}), it suffices to show that
$w^2 |g_{\alpha,n}''(w)| \to0$ uniformly in $w \in2 \overline{J}$.
For this task, setting $h^0_{l,n}(w):=h_{l,n}(w)-1$,
we have that $h'_{l,n}(w)=-h^0_{l+1,n}(w)$ and consequently,
%
%
\begin{eqnarray}
\label{der2} 
w^2 g''_{\alpha,n}(w)
&=& -(\alpha+1)+\frac{w^2h^0_{\alpha
+2,n}(w)}{h_{\alpha,n}(w)}- \biggl(\frac{wh^0_{\alpha
+1,n}(w)}{h_{\alpha,n}(w)}
\biggr)^2.
\end{eqnarray}
From (\ref{h2}), we know that for $l=1,2$, uniformly in $w \in2 \overline{J}$, as $n \to\infty$,
\[
\frac{w^l h^0_{\alpha+l,n}(w)}{h_{\alpha,n}(w)} \to\frac{\Gamma
(\alpha
+l+1)}{\Gamma(\alpha+1)}
\]
and we are done since
%
%
\begin{equation}
\label{Gamma-ident} -(\alpha+1) + \frac{\Gamma(\alpha+3)}{\Gamma(\alpha+1)} - \biggl(\frac
{\Gamma(\alpha+2)}{\Gamma(\alpha+1)}
\biggr)^2 = 0.
\end{equation}\upqed
\end{pf*}

\begin{pf*}{Proof of Lemma~\ref{after}}
To prove (\ref{hh1}), fix $\delta\in(0,1)$ and
setting $\kappa_n:=n-n^{1-\delta/2}$, note that
for $w \in2 \overline{J}$
\begin{eqnarray*}
&& \bigl(1-e^{-w}\bigr) e^{-n w} \sum
_{i=\kappa_n+1}^n \biggl|\frac{L(i)i^\alpha
}{L(n)n^\alpha}-1 \biggr| e^{iw}
\\
&&\qquad \lesssim n^{-\delta/2} + \sup_{\mu\in[1-n^{-\delta/2},1]} \biggl|\frac{L(\mu n)}{L(n)}-1 \biggr|
=: \gamma_n, 
\end{eqnarray*}
$e^{-n w} h_{\alpha,\kappa_n}(-w) \lesssim e^{-n^{\delta/3}}$
and
\[
\Biggl| \bigl(1-e^{-w}\bigr) e^{-n w} \sum
_{i=\kappa_n+1}^n e^{i w} - 1 \Biggr| \lesssim
e^{-n^{\delta/2}}. 
\]
Combining these bounds, we find that for any $\alpha\in\mathbb{R}$
and $w \in
2 \overline{J}$,
%
%
\begin{equation}
\label{lastest} \biggl|\frac{w e^{-n w} h_{\alpha,n}(-w)}{L(n)n^\alpha}-1 \biggr| \lesssim \gamma_n
\end{equation}
from which (\ref{hh1}) follows, since $\gamma_n \to0$ for
any fixed slowly varying $L(\cdot)$ and $\delta> 0$.

We now confirm (\ref{hh3}) by noting that
\[
\tilde{c}_n(u,v)=h_{\alpha,n}(-u-v)/\sqrt{h_{\alpha,n}(-2u)h_{\alpha,n}(-2v)},
\]
which by (\ref{hh1}) converges as $n \to\infty$, uniformly in $u,v
\in
\overline{J}$,
to $R(u,v) \le2 (v \vee u / v \wedge u)^{-1/2}$.

Next, proceeding along the same lines as the proof of (\ref{h4}), now with
$\overline{G}_{\alpha,n}(u,v):= \tilde{c}_n(u,v)/R(u,v)$
and $g_{\alpha,n}(w):= \log[w h_{\alpha,n}(-w)]$, reduces the proof of
(\ref{hh2}) to $w^2 |g_{\alpha,n}''(w)| \to0$, uniformly in $w \in2
\overline{J}$.
To this end, it is not hard to check that (\ref{der2}) is replaced
here by
\[
w^2 g''_{\alpha,n}(w) = -1 +
\frac{w^2h^0_{\alpha+2,n}(-w)}{h_{\alpha,n}(-w)}- \biggl(\frac{wh^0_{\alpha+1,n}(-w)}{h_{\alpha,n}(-w)} \biggr)^2 = -1 + \operatorname{Var}(w H_{n,w}),
\]
where [adopting the convention $L(0) 0^\alpha=1$], for $j=0,1,\ldots,n$,
\[
\mathbb{P}(H_{n,w}=j) = \frac{L(n-j) (n-j)^{\alpha} e^{-j w}}{\sum_{k=0}^n L(n-k) (n-k)^{\alpha} e^{-k w}}.
\]
The variance of the $\operatorname{Geometric}(e^{-w}$) random variable $H_{\infty,w}$
is $\frac{1}{4} [\sinh(w/2)]^{-2}$, hence
$\operatorname{Var}(w H_{\infty,w}) \to1$ when $w \downarrow0$. Further, as
we have already seen, truncating $w H_{\infty,w}$ and
$w H_{n,w}$ at $w n^{1-\delta/2}$ changes the corresponding variances
by at most~$e^{-n^{\delta/3}}$, uniformly over $w \in2 \overline{J}$ and
from the estimates leading to (\ref{lastest}), we easily deduce that
\[
\sup_{w \in2\overline{J}, j \le n^{1-\delta/2}} \biggl|\frac{\mathbb{P}(H_{n,w}=j)}{\mathbb{P}(H_{\infty,w}=j)} - 1 \biggr| \lesssim
\gamma_n.
\]
Combining these facts, we conclude that
\[
\sup_{w \in2\overline{J}} \biggl| \frac{\operatorname{Var}(w H_{n,w})}{\operatorname{Var}(w
H_{\infty,w})} - 1 \biggr| \lesssim
\gamma_n, 
\]
thereby completing the proof of (\ref{hh2}).
\end{pf*}

We proceed with a regularity lemma that is used in the sequel
for proving Lemma~\ref{ignore} (and Lemma~\ref{heat1}).

%
\begin{lem}\label{final}
There exist finite universal constants $K_d$, such that if centered
Gaussian process $\{Z_t, t \in T\}$,
indexed on $T=[a,b]^d \subset\mathbb{R}^d$, satisfies
%
%
\begin{equation}
\label{last11} D(s,t)^2:=\mathbb{E} \bigl[ (Z_t-Z_s)^2
\bigr] \leq M^2 \|t-s\|_2^2\qquad\forall s,t\in T
\end{equation}
for some $M<\infty$, then
%
%
\begin{equation}
\label{last12} \mathbb{E} \Bigl[ \sup_{t\in T} Z_t
\Bigr] \leq K_d M |b-a|.
\end{equation}
Further, if for $d=1$ we have that $t \mapsto Z_t \in\mathcal{C}^1$ and
%
%
\begin{equation}
\label{last13} 2 (b-a)^2 \sup_{t \in T} {\mathbb{E}
\bigl[{Z'_t}^2\bigr]} \le\sup
_{t \in T} \mathbb{E} \bigl[Z_t^2\bigr],
\end{equation}
then for some universal constant $\mu>0$,
%
%
\begin{equation}
\label{last14} \mathbb{P}\Bigl(\sup_{t\in T} \{ Z_t
\} <0\Bigr)\geq\mu.
\end{equation}
\end{lem}

\begin{pf} For proving (\ref{last12}) note that there exist
$C_d<\infty$
such that $T$ is covered by at most
$N(\epsilon) = \min\{1,\epsilon^{-d}(C_d M |b-a|)^d\}$
Euclidean balls of radius $\epsilon/M$. With $B_D(s,r)=\{t \in T\dvtx
D(s,t) \le\epsilon\}$ denoting
the ball in pseudo-metric $D(\cdot,\cdot)$ of radius $\epsilon\ge0$
and center $s \in T$ and
$B(s,\epsilon)$ the Euclidean ball of same radius and center,
our assumption (\ref{last11}) implies that $B(s,\epsilon/M) \subseteq
B_D(s,\epsilon)$ for any $s \in T$, thereby
inducing a cover of $T$ by at most $N(\epsilon)$ balls of radius
$\epsilon$ in pseudo-metric $D(\cdot,\cdot)$.
Recall \cite{AT},  Theorem 1.3.3,
that there exist universal finite $K_0$ such that
\[
\mathbb{E} \Bigl[ \sup_{t \in T} Z_t \Bigr] \leq
K_0\int_{0}^{C_d M |b-a|} \sqrt{\log N(\epsilon)}\,d\epsilon.
\]
Our\vspace*{-2pt} thesis follows upon change of
variable $y=\sqrt{d^{-1} \log N(\epsilon)}$,
with $K_d:=2\sqrt{d} K_0 C_d \int_{0}^\infty y^2 e^{-y^2} \,dy$.

Turning to prove (\ref{last14}), let
${\sigma_T}^2:= \sup_{t\in T} \mathbb{E}[Z_t^2]$ and
$\overline{Z}_t:= Z_t-Z_{t_0}$ for $t_0\in T$
such that $\mathbb{E}[Z_{t_0}^2] = \sigma_T^2$.
Then, by Cauchy--Schwarz
we have that for any $s,t \in T$,
\[
\mathbb{E}\bigl[ (\overline{Z}_t-\overline{Z}_s)^2
\bigr] = \mathbb{E}\bigl[ (Z_t-Z_s)^2
\bigr] \leq (t-s)^2 \sup_{u\in[s,t]}\mathbb{E}\bigl[
Z_u^{\prime 2}\bigr].
\]
Thus, (\ref{last13}) results with
\[
\bar{\sigma}_T^2:= \sup_{t\in T}
\mathbb{E}\bigl[\overline{Z}_t^2\bigr] \leq
\frac{1}{2} {\sigma_T}^2
\]
and considering (\ref{last12}) for $\overline{Z}_t$, we further have that
$\mathbb{E} [\sup_{t\in T} \overline{Z}_t ] \leq K_1 \sigma_T$.
Clearly,
\[
\sup_{t \in T} Z_t = Z_{t_0} + \sup
_{t\in T} \overline{Z}_t,
\]
so by a union bound we have for any $\lambda>0$,
%
%
\begin{equation}
\label{eqamir-0} \mathbb{P}\Bigl(\sup_{t\in T} \{ Z_t
\} <0\Bigr)\geq \mathbb{P}(Z_{t_0}<-\lambda\sigma_T)-
\mathbb{P}\Bigl(\sup_{t\in T}\{ \overline{Z}_t \} >
\lambda\sigma_T\Bigr).
\end{equation}
For $\lambda\ge K_1$, large enough the first term
on the right-hand side is at least $0.5 e^{-\lambda^2/2}$ and
by Borell-TIS inequality, the second term is at most
\[
2 \exp \biggl\{-\frac{(\lambda-K_1)^2 \sigma_T^2}{2 \bar{\sigma}_T^2} \biggr\} \le2 e^{-(\lambda-K_1)^2}.
\]
This completes the proof, since
$\mu:= 0.5 e^{-\lambda^2/2} - 2 e^{-(\lambda-K_1)^2}$ is strictly positive
for~$\lambda$ large enough.
\end{pf}

We establish part (a) of Lemma~\ref{ignore} by partitioning
relevant domains of $Q_n(e^{-\cdot})$ to at most
$\gamma(\delta) T_n$ subintervals, within each of
which (\ref{last13}) holds [and where $\gamma(\delta) \to0$],
thereby combining Lemma~\ref{final} and Slepian's inequality.
However, to provide the estimates of part (b) in
\textit{critical case} of $\alpha=-1$, we require
the following comparison (after a change of argument),
between $Q_n(e^{-\cdot})$ and the standard
stationary Ornstein--Uhlenbeck process $\{X_t,t\geq0\}$.

%
\begin{lem}\label{O-U}
For $\alpha= -1$ and any slowly varying $L(\cdot)$,
there exist $r(\gamma) \downarrow0$ when $\gamma\downarrow0$, such that
%
%
\begin{equation}
\label{eqsumit-1} \bar{c}_n(u,v)\geq \biggl(\frac{u}{v}
\biggr)^{r(\gamma)}\qquad\forall0 < u \leq v \leq\gamma.
\end{equation}
\end{lem}

\begin{pf} First note that for $v \ge u \ge0$, by
the monotonicity of $u \mapsto h_{\alpha,n}(u)$,
\[
\bar{c}_n(u,v)=\frac{h_{\alpha,n}(u+v)}{\sqrt{h_{\alpha,n}(2u)h_{\alpha,n}(2v)}} \geq\frac{h_{\alpha,n}(2v)}{h_{\alpha,n}(2u)} \geq
\frac{h_{\alpha,\infty}(2v)}{h_{\alpha,\infty}(2u)},
\]
where the second inequality follows by noting that
$n \mapsto h_{\alpha,n}(2v)/h_{\alpha,n}(2u)$
is monotone decreasing
[for $e^{-2(n+1)(v-u)} \leq h_{\alpha,n}(2v)/h_{\alpha,n}(2u)$
via term by term comparison]. We thus get (\ref{eqsumit-1})
upon finding $r = r(\gamma) \downarrow0$ for which
$\xi_{r} (u):= u^r h_{-1,\infty}(u)$
is nondecreasing on $(0,2 \gamma]$. Since $\xi'_r (u) \ge0$
if and only if
\[
r \ge\zeta(u):= \frac{u h^0_{0,\infty}(u)}{h_{-1,\infty}(u)},
\]
this amounts to showing that $\zeta(u) \downarrow0$ for $u \downarrow0$.
To this end, recall (\ref{estimate}) that
$u h^0_{0,\infty} (u) \lesssim L(1/u)$ and moreover for any $\eta>0$,
\[
h_{-1,\infty}(u) \ge e^{-1} \sum_{i=\eta/u}^{1/u}
L(i) i^{-1} \ge e^{-1} L(1/u) \bigl(1+o(1)\bigr) \log(1/\eta),
\]
so considering $u \downarrow0$ followed by $\eta\downarrow0$
we conclude that also $\zeta(u) \to0$ as $u \downarrow0$.
\end{pf}

\begin{pf*}{Proof of Lemma~\ref{ignore}}
\begin{longlist}[(a)]
\item[(a)] We first consider $\alpha> -1$ and
establish (\ref{eqlast}) by partitioning
$[-n^{-(1-\delta)},\break n^{-(1-\delta)}]$ to at most $\gamma(\delta) T_n$
intervals $\{I_k\}$, with
$\gamma(\delta) \to0$, such that $Z_u = e^{n (u \wedge0)} Q_n(e^{-u})$
satisfies (\ref{last13}) within each such subinterval $I_k$. Indeed,
since $Q_n(e^{-u})$ has nonnegative autocorrelation, by Slepian's
inequality and (\ref{last14}) we have then that
\[
\mathbb{P}\Bigl(\sup_{|u| \le n^{-(1-\delta)}} \bigl\{ Q_n
\bigl(e^{-u}\bigr) \bigr\} < 0\Bigr) \geq\prod
_k \mathbb{P}\Bigl(\sup_{u \in I_k} \{
Z_u \} < 0\Bigr) \geq\mu^{\gamma(\delta) T_n}
\]
for some universal constant $\mu>0$, yielding (\ref{eqlast}) upon
considering $T_n^{-1} \log(\cdot)$ of these probabilities
in the limit $n \to\infty$ followed by $\delta\downarrow0$.

To carry out this program, note first that both
$\mathbb{E}[Q_n(e^{-u})^2]=h_{\alpha,n}(2u)$ and
$\mathbb{E}[Q'_n(e^{-u})^2]=h_{\alpha+2,n}(2u)$
are monotone in $u \ge0$, with (\ref{last13}) obviously
satisfied within \textit{any} subinterval of size $1/(2n)$.

Further, from (\ref{estimate}) we have that for any $l>-1$ there
exist finite $b=b_l$ and positive
$w_l$,
so that $u^{l+1} h_{l,n}(u)/L(1/u)$ is bounded (and bounded
away from zero), uniformly in $u \in[0,w_l]$ and $n \ge b_l/u$.
So, with $\alpha>-1$, the same
applies for
$u^2 h^0_{\alpha+2,n}(2u)/h_{\alpha,n}(2u)$.
This in turn
implies that for some $\eta>0$, $u_\star>0$ and $b \ge2$
finite [depending only on $\alpha$ and $L(\cdot)$],
setting $u_{k,n}=k/(2n)$, $k=0,\ldots,b$
and $u_{k+b,n}=b e^{\eta k}/(2n)$, $k \ge0$,
the process $Z_u=Q_n(e^{-u})$
satisfies (\ref{last13}) in each interval
$I_k=[u_{k-1,n},u_{k,n}]$, $k \ge1$, provided $u_{k,n} \le u_\star$.
Since $u_{k_\star+ b,n} \ge n^{-(1-\delta)}$ for
$k_\star:= (\delta/\eta) T_n$, this takes care of
the part of $u \ge0$ in (\ref{eqlast}). In case
$u=-w<0$, we follow the same reasoning, just now
applying Lemma~\ref{final} for the rescaled process
$Z_w:= e^{-n w} Q_n(e^w)$, $w \ge0$. Specifically,
setting
\[
\tilde{h}_{l,n}(w):= \sum_{j=0}^n
L(n-j) (n-j)^\alpha j^l e^{-j w}
\]
for $l=0,2$ [with $L(0) 0^\alpha:=1$], it is
easy to check that $\mathbb{E}[Z_w^2]=\tilde{h}_{0,n}(2w)$
and $\mathbb{E}[{Z'_w}^2]=\tilde{h}_{2,n}(2w)$. Thus,
per $\alpha> -1$ and slowly varying $L(\cdot)$, the same
partition takes care of $u < 0$ in (\ref{eqlast})
provided
$w^{3} \tilde{h}_{2,n}(w)/(L(n) n^{\alpha})$
is bounded and
$w \tilde{h}_{0,n}(w)/(L(n) n^{\alpha})$
bounded away from zero,
uniformly in $w \in[b n^{-1}, w_\star]$,
for some $b<\infty$ and $w_\star>0$.
To this end, fixing $l \ge0$ and $\epsilon\in(0,1)$,
note that the ratio between
$\sum_{j \le(1-\epsilon) n} L(n-j) (n-j)^\alpha j^l e^{-j w}$
and $L(n) n^\alpha\sum_{j \le(1-\epsilon) n} j^l e^{-j w}$
is bounded and bounded away from zero, uniformly in $n$ and $w$
(for any $\alpha\in\mathbb{R}$), and the same applies for the ratio between
the latter and $L(n) n^{\alpha}/w^{l+1}$, provided
$(1-\epsilon) (n w) \ge b$ [as shown in the course
of proving (\ref{h2})]. Next, recall that
$\sum_{i=0}^n L(i) i^{\alpha} \lesssim L(n) n^{\alpha+1}$
for $\alpha> -1$ and slowly varying $L(\cdot)$;
hence, we are done, for
\begin{eqnarray*}
\sum_{j > (1-\epsilon) n}^n L(n-j) (n-j)^\alpha
j^l e^{-j w} &\le& e^{-(1-\epsilon) n w} n^l \sum
_{i=0}^{\epsilon n} L(i) i^\alpha
\\
&\lesssim&
L(n) n^{\alpha} w^{-(l+1)} \xi_\epsilon(n w),
\end{eqnarray*}
where $\xi_\epsilon(b):= b^{l+1} e^{-(1-\epsilon) b} \to0$
as $b \to\infty$.

Having dealt with (\ref{eqlast}) for $\alpha>-1$, we turn
to $\alpha\leq-1$ and fixing $\gamma>0$
set $b(\gamma): = \gamma-(\alpha+1)$. Fixing $l \ge0$,
we claim that $w^{l+1} \tilde{h}_{l,n}(w)/(L(n) n^{\alpha})$
is bounded and bounded away from zero, uniformly
in $w \in[b(\gamma) T_n n^{-1}, w_\star]$. Indeed, the only difference
is that now $\sum_{i=0}^n L(i) i^\alpha\lesssim L(n) n^{\eta}$
for any fixed $\eta>0$, so to neglect the contribution
of $j > (1-\epsilon) n$ to $\tilde{h}_{l,n}(w)$ we need that
\[
n^{\eta- (\alpha+1)} \xi_\epsilon(n w) \to0,
\]
which applies for any $n w \ge b(\gamma) T_n$ if
$\epsilon>0$ and $\eta>0$ are small enough
so that $\gamma(1-\epsilon) > 2 \eta- \epsilon(\alpha+1)$.
We further cover $[0,\gamma T_n/(2n)]$ and
$[b(-\gamma) T_n/(2n),\break b(\gamma) T_n/(2n)]$
by at most $3 \gamma T_n$ intervals of equal length $1/(2n)$,
within each of which Lemma~\ref{final} applies for
$Z_w = e^{-nw} Q_n(e^w)$. So, given that (\ref{eqfirst-b})
handles the domain $u \ge0$, by the same reasoning as before,
we establish (\ref{eqlast})
by showing that for any fixed $\gamma> 0$,
$\alpha< -1$ and $\eta>0$ small enough, the process
$w \mapsto Q_n(e^w)$ satisfies condition (\ref{last13})
within each subinterval of the partition
of $[\gamma T_n/(2n), b(-\gamma) T_n/(2n)]$
given by $w_{k,n}=e^{\eta k} w_{0,n}$, $k \ge1$,
and $w_{0,n}=\gamma T_n/(2n)$. As $h_{\alpha,n}(-w) \ge1$,
this in turn amounts to proving that $w^2 h^0_{\alpha+2,n}(-w)$
is uniformly bounded on $(0, b(-\gamma) T_n/n]$.
Indeed, adapting the calculation leading to
(\ref{lastest}), now for $\kappa_n = \epsilon n$ and
with $L(i) \lesssim i^{\epsilon}$, we find that
\[
h^0_{\alpha+2,n}(-w) \lesssim e^{nw}
n^{\epsilon+ \alpha+3} + e^{\epsilon n w} n^{\epsilon+ (\alpha+3)_+},
\]
which yields the stated uniform boundedness
for $e^{n w} \le n^{b(-\gamma)}$ upon choosing
\mbox{$\epsilon>0$} small enough so that
\[
b(-\gamma) + \epsilon+ \alpha+ 1 < 0, \qquad \epsilon b(-\gamma) + \epsilon+ (
\alpha+ 3)_+ - 2 < 0.
\]

We proceed to confirm (\ref{eqfirst-a}) where,
by (\ref{eqfirst-b}), if $\alpha\le-1$ we
only need to consider $u = -w \le0$. Setting
$w_{k,n}:= e^{\eta k} n^{-\delta}$, $k \ge0$,
recall that we have already seen that for any
$\alpha\in\mathbb{R}$ and $\eta>0$ small enough, the
rescaled process $Z_w$ satisfies~(\ref{last14})
within each subinterval $I_k:=[w_{k-1,n},w_{k,n}]$
[and when $\alpha>-1$ the same applies also for
$Z_u = Q_n(e^{-u})$ with $u >0$].
Hence, partitioning $\pm u \in[n^{-\delta},u_\star]$
for fixed $u_\star\in(0,1]$ to at most $k_\star$ such
subintervals, by the same reasoning we applied for~(\ref{eqlast}) in case $\alpha> -1$, the proof of
(\ref{eqfirst-a}) reduces to showing that for
all $\alpha\in\mathbb{R}$ and any fixed $u_\star>0$,
%
%
\begin{equation}
\label{eqfirst-am} \inf_{n} \mathbb{P}\bigl(Q_n
\bigl(e^{-u}\bigr) < 0,\ \forall|u| \ge u_\star\bigr) > 0.
\end{equation}
We deal with $u \le- u_\star$ in (\ref{eqfirst-am}) by
equivalently, considering
$\{R_n(x):= x^n Q_n(x^{-1}) < 0\}$ for $x \in(0,x_\star]$,
with $x_\star:= e^{-u_\star} < 1$. Specifically, note
that for $x \in[0,x_\star]$,
\[
\mathbb{E}\bigl[ R'_n (x)^2\bigr] \lesssim
\sum_{j=2}^{n} L(n-j) (n-j)^{\alpha
}
j^2 x_\star^{2j}
\]
is bounded by $C L(n) n^\alpha$ for $C=C(\alpha,L(\cdot))$ finite and
all $n$. Indeed, with $\sum_{j=0}^\infty j^2 x_\star^{2j}$ finite, such
bound applies for the sum over $j \le(1-\epsilon) n$ on the
right-hand side,
whereas the remainder sum over $(1-\epsilon) n < j \le n$ contributes
at most
\[
n^2 x_\star^{2 (1-\epsilon) n} \sum
_{i=0}^{\epsilon n} L(i) i^\alpha,
\]
which is exponentially decaying in $n$, hence dominated by $L(n)
n^\alpha$.
Since $\mathbb{E}[R_n(x)^2] \ge L(n) n^\alpha$ for all $x > 0$ and $n$,
the uniform partition of $[0,x_\star]$ to
$r$ subintervals $\{I_k\}$ of length
$x_\star/r$ each, results for $r$ large enough with $x \mapsto R_n(x)$
satisfying (\ref{last13}) within each subinterval $I_k$.
Hence, by Slepian's inequality,
we get that
$\mathbb{P} (\sup_{x \in[0,x_\star]} \{ R_n(x) \} < 0) \ge\mu^r$.
The same argument applies for $u \ge u_\star$, since
$\mathbb{E}[Q_n(x)^2] \ge1$ for all $x \ge0$ and
\[
\mathbb{E}\bigl[Q_n'(x)^2\bigr] \le\sum
_{i=1}^{\infty} L(i) i^{\alpha+2}
x^{2(i-1)}
\]
is\vspace*{1pt} uniformly bounded on $[0,x_\star]$ [for any
fixed $\alpha\in\mathbb{R}$ and slowly varying $L(\cdot)$].

\item[(b)]
Setting $v_n:= \mathbb{E} [Q_n(1)^2] = 1 + \sum_{i=1}^n L(i) i^{\alpha}$
and $\overline{Q}_n(x):= Q_n(x) - Q_n(1)$, note that
\[
\sup_{x \in[0,1]} \mathbb{E} \bigl[ \overline{Q}_n(x)^2
\bigr] = v_n - 1.
\]
If the monotone limit $v_\infty$ of $v_n$ is finite, then
$x \mapsto Q_\infty(x) = \sum_{i=0}^{\infty} a_i x^i$ is a
well-defined centered Gaussian process on $[0,1]$ whose
sample path are a.s. (uniformly) continuous; hence,
$K_\infty:= \mathbb{E} [\sup_{x \in[0,1]} Q_\infty(x)]$
is finite. Since $n \mapsto\mathbb{E}[(Q_n(x)-Q_n(y))^2]$
is nondecreasing, it follows from Sudakov--Fernique inequality
that the (nondecreasing) sequence
$K_n:=\mathbb{E} [ \sup_{x \in[0,1]} Q_n(x) ]$ is
bounded above by $K_\infty$. As argued around (\ref{eqamir-0}),
by Borell-TIS inequality, for any
$\lambda\ge K_\infty\ge\sup_n K_n$ large enough and all $n$,
\begin{eqnarray*}
p_{[0,1]}(n) &\ge& 
\mathbb{P}\bigl(Q_n(1) < -
\lambda\sqrt{v_n}\bigr) - \mathbb{P}\Bigl(\sup_{x \in[0,1]}
\bigl\{ \overline{Q}_n(x) \bigr\} > \lambda\sqrt{v_n}
\Bigr)
\\
&\geq& 0.5 e^{-\lambda^2/2} - 2 e^{-(\lambda-K_n)^2 v_n/(2 (v_n-1))},
\end{eqnarray*}
with $v_n \uparrow v_\infty\in[1,\infty)$, the right-hand side
is bounded away from zero for some $\lambda$ and all $n$
large enough, and hence so is $n \mapsto p_{[0,1]}(n)$.

Assuming hereafter that $v_\infty=\infty$ and in particular
that $\alpha= -1$, in view of Lemma~\ref{O-U}, we get
(\ref{eqfirst-b}) once we show that
%
%
\begin{equation}
\label{eqfirst-bm} \liminf_{n\rightarrow\infty} \frac{1}{T_n}\log\mathbb{P}\Bigl( \sup
_{u \in[\gamma n^{-1},\gamma]} \bigl\{ Q_n\bigl(e^{-u}\bigr)
\bigr\} < 0\Bigr) \ge- r (\gamma)
\end{equation}
(which per Lemma~\ref{O-U} converges to zero as $\gamma\downarrow0$).
This is done upon realizing that the auto-correlation function of
$u \mapsto X_{-2 r(\gamma) \log(u/\gamma)}$ matches the
right-hand side of (\ref{eqsumit-1}), hence by Slepian's inequality,
\[
\mathbb{P}\Bigl( \sup_{u \in[\gamma n^{-1},\gamma]} \bigl\{ Q_n
\bigl(e^{-u}\bigr) \bigr\} < 0\Bigr) \ge \mathbb{P}\Bigl(\sup
_{t \in[0, 2 r (\gamma) T_n]} \{ X_t \} < 0\Bigr)
\]
\end{longlist}
and (\ref{eqfirst-bm}) follows, since
$X_t$ has persistence exponent $1/2$.
\end{pf*}

\section{Proof of Theorem \texorpdfstring{\protect\ref{Heat}}{1.5}}\label{sec-5}

We start with two lemmas, the first of which provides
for each fixed positive time a smooth initial condition of
the required law, while the second explicitly constructs
a solution of the heat equation for such initial condition.

%
%
\begin{lem}\label{heat1}
Equip $\mathcal{ A}=\mathcal{{C}}({\mathbb{R}}^{d})$
with the topology of uniform convergence on compact sets.
For any $\varepsilon>0$, there exists an
$(\mathcal{ A},\mathcal{ B}_\mathcal{ A})$-valued
centered Gaussian field $g_\varepsilon(\cdot)$
with covariance $C_\varepsilon({\mathbf x}_1,{\mathbf x}_2) = K_{2 \varepsilon
}({\mathbf x}_1-{\mathbf x}_2)$
such that $|g_\varepsilon({\mathbf x})|\leq a\|{\mathbf x}\|+b$ for some
$a,b$ (possibly random) and all ${\mathbf x}$.
\end{lem}

\begin{pf}
Since $C_\varepsilon(\cdot,\cdot)$ is positive definite, there
exists a centered
Gaussian field $g_\varepsilon({\mathbf x})$ indexed on $\mathbb{R}^d$ with
covariance function $C_\varepsilon(\cdot,\cdot)$.
Further, with $\delta=2\varepsilon$ and utilizing the bound $1-e^{-r}
\le r$,
%
%
\begin{equation}
\label{eqg-cov} \mathbb{E} \bigl[ \bigl(g_\varepsilon({\mathbf
x}_1)-g_\varepsilon({\mathbf x}_2)\bigr)^2
\bigr] = 2 \bigl( K_{\delta} ({\mathbf0}) - K_{\delta}({\mathbf
x}_1-{\mathbf x}_2)\bigr) \leq\frac{\|{\mathbf x}_1-{\mathbf x}_2\|^2}{(4 \pi\delta)^{d/2}2
\delta}.
\end{equation}
Hence, using the induced bound on higher moments of
$g_\varepsilon({\mathbf x}_1)-g_\varepsilon({\mathbf x}_2)$, by
Kolmogorov--Centsov continuity
theorem we can and shall consider hereafter the unique continuous
modification of $g_\varepsilon(\cdot)$, which takes values in
${\mathcal A}$
and is measurable with respect to the corresponding
Borel $\sigma$-algebra ${\mathcal B}_{\mathcal A}$.

Combining the bound (\ref{eqg-cov}) with Lemma~\ref{final},
we have that\break
$\mathbb{E}[ \sup_{\|{\mathbf x}\|\leq n} g_\varepsilon({\mathbf x})
] \leq M' n$,
for some finite $M'=M'(d,\eta)$ and all $n$.
Further, with $\mathbb{E}[g_\varepsilon({\mathbf x})^2]=K_{2 \varepsilon
}({\mathbf0})$ uniformly
bounded in ${\mathbf x}$, we have by Borell-TIS inequality and the
symmetry of $g_\varepsilon(\cdot)$, that
\[
\mathbb{P} \Bigl( \sup_{\|{\mathbf x}\|\leq n} \bigl|g_\varepsilon({\mathbf x})\bigr| > 2
M' n \Bigr) \leq2 e^{-{M'^2 n^2}/{2 K_{2 \varepsilon}({\mathbf0})}}.
\]
Hence, by the Borel--Cantelli lemma, almost surely
$\sup_{\|\mathbf x \|\leq n} |g_\varepsilon({\mathbf x})| \leq2 M' n$
for all $n \ge N(\omega)$ large enough,
so $|g_\varepsilon({\mathbf x})|\leq a\|{\mathbf x}\|+b$, for $a=2 M'$ and
$b= b(\omega) =\sup_{\|{\mathbf x}\| \leq N(\omega)} |g_\varepsilon
({\mathbf x})|$
is a.s. finite [since $N(\omega)$ is a.s. finite and $g_\varepsilon
\in\mathcal{A}$].
Finally, to have such growth condition hold for \textit{all} $\omega$,
let $g_\varepsilon(\cdot) \equiv0$ on the null set where $N(\omega
)=\infty$,
which neither affects the law of $g_\varepsilon(\cdot)$ nor its
sample path
continuity.
\end{pf}

%
\begin{lem}\label{heat2}
Let $g\in\mathcal{ A}$ satisfy
$|g({\mathbf x})|\leq a\|{\mathbf x}\|+b$ for some $a,b$ finite.
Then, for any $d=1,\ldots,$ and $\varepsilon> 0$, setting
$\mathbb{D}_\varepsilon= \mathbb{R}^d \times(\eta,\infty)$, the function
%
%
\begin{equation}
\label{eqsol-heat} \phi({\mathbf x},t)= \int_{\mathbb{R}^d} K_{t-\varepsilon}({\mathbf x}-{\mathbf y})g({
\mathbf y})\,d{\mathbf y} = \int_{\mathbb{R}^d} K_{t-\varepsilon}({\mathbf y})g({\mathbf x} - \mathbf{y})\,d{\mathbf
y}
\end{equation}
is a solution in $\mathcal{C}_\varepsilon:= \mathcal{C}^{2,1}
(\mathbb{D}_\varepsilon)$
of the heat equation (\ref{E1}), and the unique such
solution which converges to $g({\mathbf x})$ for $t \downarrow\eta$
and satisfies the growth condition
$|\phi({\mathbf x},t)|\leq p\|{\mathbf x}\|+q\sqrt{t}+r$ for
some finite constants $p,q,r$.
\end{lem}

\begin{pf} Since $K_{s}(\cdot)$ is a probability density
on $\mathbb{R}^d$ such that $\int\|{\mathbf u}\|^2\* K_s ({\mathbf u}) \,d{\mathbf u}
= 2 d s$,
from the given growth condition of $g(\cdot)$ it follows that for
any $t>\eta$,
\[
\bigl|\phi({\mathbf x},t)\bigr|\leq b+a \|{\mathbf x}\| + a \int_{\mathbb{R}^d} \|{\mathbf y}\| K_{t-\varepsilon}
({\mathbf y}) \,d{\mathbf y} \leq b + a \|{\mathbf x}\| + a \sqrt{2d (t-\varepsilon)}.
\]
Thus, $\phi(\cdot,\cdot)$ of (\ref{eqsol-heat}) is well defined and
satisfies the growth condition (with $p=a$, $q=a\sqrt{2d}$ and $r=b$).
With $\phi({\mathbf x},\varepsilon+s)$ alternatively being the
expected value of $g({\mathbf x} - \sqrt{s} {\mathbf U})$ for a
standard multivariate normal ${\mathbf U}$, dominated convergence
provides its convergence to $g({\mathbf x})$ (uniformly on compacts),
as $s \downarrow0$.

To confirm that $\phi\in\mathcal{C}_\varepsilon$ satisfies
the heat equation (\ref{E1}) on $\mathbb{D}_\varepsilon$, note that
\begin{eqnarray*}
\phi({\mathbf x},t)&=& K_{t-\varepsilon}({\mathbf x}) F \biggl( \frac{\mathbf x}{2(t-\varepsilon)},
\frac{1}{4(t-\varepsilon)} \biggr),
\\
F({\bolds\theta}_1, \theta_2)&:=& \int_{\mathbb{R}^d} e^{{\bolds\theta}_1'{\mathbf y} -\theta_2 {\mathbf y}'\mathbf{y}} g({\mathbf y}) \,d{\mathbf y}.
\end{eqnarray*}
Clearly, $K_t({\mathbf x}) \in\mathcal{C}^\infty(\mathbb{D}_0)$ and combining
the assumed linear growth of $g(\cdot)$ with dominated
convergence, we have that also $F \in\mathcal{C}^\infty(\mathbb
{D}_0)$. Hence,
$\phi\in\mathcal{C}_\varepsilon$ and by the same reasoning, each partial
derivative of $\phi(\cdot,\cdot)$ can be taken
within the integral~(\ref{eqsol-heat}) over ${\mathbf y}$.
As $K_t({\mathbf x})$ satisfies (\ref{E1}) on $\mathbb{D}_0$,
it thus follows that $\phi(\cdot)$ satisfies this
PDE on $\mathbb{D}_\varepsilon$. Finally, the uniqueness of solution
of (\ref{E1}) in
$\mathcal{C}_\varepsilon$ subject to the assumed linear growth condition
and the given initial condition $g \in\mathcal{A}$ at $t=\varepsilon
$, is
well known (e.g., see \cite{Evans}, Theorem 2.3.7, for
uniqueness on $[\varepsilon,T]$, any $T>0$).
\end{pf}

We now complete the proof of Theorem~\ref{Heat} by
combining the preceding lemmas with Kolmogorov's
extension theorem (to construct one measurable
solution over all of $\mathbb{D}_0$).

\begin{pf*}{Proof of Theorem~\ref{Heat}}
Fixing $\delta=2\varepsilon>0$, by Lemma~\ref{heat1}
there exists centered $(\mathcal{ A},\mathcal{ B}_\mathcal{ A})$-valued
Gaussian field $g_\varepsilon(\cdot)$ of law $\mathbb{P}_\varepsilon
$ corresponding
to covariance function $K_{\delta}({\mathbf x}_1-{\mathbf x}_2)$.
We claim that $\phi|_\varepsilon=\mathbb{T}_\varepsilon
(g_\varepsilon)$ given by (\ref{eqsol-heat})
for $t \ge\delta$, is $(\mathcal{C}_\delta,\mathcal{B}_{\mathcal
{C}_\delta})$-measurable.
Indeed, consider smooth $\psi\dvtx \mathbb{R}\mapsto[0,1]$ supported
on $\mathbb{R}_+$ such that $\psi(r) =1$ for $r \ge1$ and let
$\hat{\phi}_n = \mathbb{T}_{\varepsilon,n}(g_\varepsilon)$,
given by
\[
\hat{\phi}_n ({\mathbf x},t) = \int_{\mathbb{R}^d} \psi\bigl(n - \|{\mathbf
x} - \mathbf{y}\|^2\bigr) K_{t-\varepsilon}({\mathbf x} - \mathbf{y}) g_\varepsilon({
\mathbf y}) \,d{\mathbf y}.
\]
Since these integrals are over bounded domains of ${\mathbf y}$ values
and $({\mathbf x},t) \mapsto K_{t-\varepsilon}({\mathbf x}) \psi(n - \|{\mathbf
x}\|^2)$
is smooth for $t \ge\delta>\varepsilon$, each mapping
$\mathbb{T}_{\varepsilon,n}\dvtx (\mathcal{A},\mathcal{B}_\mathcal{A})
\mapsto(\mathcal{C}_\delta,\mathcal{B}_{\mathcal{C}_\delta})$
is continuous (with respect to the relevant uniform convergence on compacts).
Further, by the growth condition of Lemma~\ref{heat1} on
$g_\varepsilon$,
for any $M<\infty$ and multi-index $({\mathbf r},\ell)$,
\begin{eqnarray*}
&& \sup_{\|{\mathbf x}\|\leq M, s \in[0,M]} \biggl| \frac{\partial}{\partial x_{r_1} \cdots\partial x_{r_k}\,\partial
s^\ell} \int_{\mathbb{R}^d} K_{s+\varepsilon}({\mathbf x}- \mathbf{y}) \bigl(1-\psi \bigl(n-\|{\mathbf x}- \mathbf{y}\|^2\bigr)\bigr)
g_\varepsilon({\mathbf y}) \,d{\mathbf y} \biggr|
\\
&&\qquad \stackrel{n \to\infty} {\longrightarrow} 0.
\end{eqnarray*}
Consequently, we have that $\mathbb{T}_{\varepsilon,n}(g_\varepsilon
) \to\phi|_\varepsilon$
in $\mathcal{C}_\delta$ as $n\to\infty$, yielding the
Borel measurability of $\phi|_\varepsilon$.

Let $\mathbb{Q}_{\delta}=\mathbb{P}_{\varepsilon} \circ\mathbb
{T}_{\varepsilon}^{-1}$ denote the
centered Gaussian law of $\phi|_\varepsilon$ thus induced on
$(\mathcal{C}_\delta,\mathcal{B}_{\mathcal{C}_\delta})$ by (\ref{eqsol-heat}).
For any $\delta'>\delta\ge0$, clearly $\mathbb{D}_{\delta} \subset
\mathbb{D}_{\delta'}$
making the identity map a projection
$\pi_{\delta,\delta'}\dvtx  \mathcal{C}_{\delta} \mapsto\mathcal
{C}_{\delta'}$,
with the complete, separable, metrizable
space $\mathcal{C}_0$ being homeomorphic to the projective limit of
$\{\mathcal{C}_\delta, \delta>0\}$ (with respect to these projections).
It is easy to check that for all $t,s \ge\delta$,
\begin{eqnarray*}
&& \mathbb{E}\bigl[ \phi\bigl|_\varepsilon({\mathbf x}_1,t)
\phi\bigr|_\varepsilon({\mathbf x}_2,s) \bigr]
\\
&&\qquad  = \int\!\!\int K_{t-\varepsilon}({\mathbf x}_1-{\mathbf y}_1)
K_{s-\varepsilon
}({\mathbf x}_2-{\mathbf y}_2)
C_\varepsilon({\mathbf y}_1,{\mathbf y}_2) \,d {\mathbf
y}_1 \,d {\mathbf y}_2
\\
&&\qquad = K_{t+s}({\mathbf
x}_1-{\mathbf x}_2),
\end{eqnarray*}
is independent of $\varepsilon>0$. In particular, for any
$\delta'>\delta>0$ the Borel probability measure $\mathbb{Q}_{\delta'}$
on $\mathcal{C}_{\delta'}$ is just the push-forward of $\mathbb
{Q}_{\delta}$
via the projection $\pi_{\delta,\delta'}$. Consequently, setting
the f.d.d. of $\{ \phi|_{\varepsilon'} (\cdot)\dvtx  \varepsilon' \ge
\varepsilon\}$
on $(0,\infty)$ to match those of $\{ \pi_{2\varepsilon,2\varepsilon
'}(\phi|_\varepsilon)\dvtx
\varepsilon' \ge\varepsilon\}$
yields a consistent collection, so Kolmogorov's
extension theorem provides existence of Borel probability
measure $\mathbb{Q}_0$ on $\mathcal{C}_0$ such that each $\mathbb
{Q}_\delta$
is the push-forward of $\mathbb{Q}_0$ by $\pi_{0,\delta}$
(see, e.g., \cite{D}, Theorems 12.1.2 and 13.1.1).
In particular, $\mathbb{Q}_0$ corresponds to a centered
Gaussian field $\phi_d \in\mathcal{C}_0$ having the same
covariance as its restrictions $\phi|_\varepsilon$
to subdomains $\mathbb{D}_{2 \varepsilon}$. As each $\phi
|_\varepsilon$
satisfies (\ref{E1}) on $\mathbb{D}_\varepsilon$, clearly
$\phi_d$ satisfies it throughout $\mathbb{D}_0$ and
the identity (\ref{E4}) further follows from
our explicit construction via (\ref{eqsol-heat})
of the restriction of $\phi_d$ to $\mathbb{D}_{t_1}$
[by utilizing Fubini's theorem, the growth condition of Lemma~\ref{heat1}
and convolution properties of the Brownian semigroup $t \mapsto
K_t(\cdot)$].
Finally, $\phi_d \in\mathcal{C}^\infty(\mathbb{D}_0)$ by the
integral representation
(\ref{E4}) and smoothness of $({\mathbf x},t) \mapsto K_t({\mathbf x})$.
\end{pf*}

\section*{Acknowledgments}
We thank Jonathan Taylor for many helpful discussions
and the anonymous referees whose comments helped to
improved the presentation of this paper.



%

\printaddresses

\end{document}